\documentclass[twoside, 12pt]{article}
\usepackage{amssymb, amsmath, mathrsfs, amsthm}
\usepackage{graphicx}
\usepackage{color}
\usepackage{enumerate}
\usepackage{tabularx}

\usepackage[top=2cm, bottom=2cm, left=2cm, right=2cm]{geometry}
\usepackage{float, caption, subcaption}

\input{epsf}

\newcommand{\bd}{\begin{description}}
\newcommand{\ed}{\end{description}}
\newcommand{\bi}{\begin{itemize}}
\newcommand{\ei}{\end{itemize}}
\newcommand{\be}{\begin{enumerate}}
\newcommand{\ee}{\end{enumerate}}
\newcommand{\beq}{\begin{equation}}
\newcommand{\eeq}{\end{equation}}
\newcommand{\beqs}{\begin{eqnarray*}}
\newcommand{\eeqs}{\end{eqnarray*}}

\definecolor{DarkGreen}{rgb}{0.2, 0.6, 0.3}

\newtheorem{theorem}{Theorem}[section]
\newtheorem{conjecture}{Conjecture}[section]

\newtheorem{lemma}{Lemma}[section]
\newtheorem{remark}{Remark}[section]

\newtheorem{definition}{Definition}
\newtheorem{corollary}[theorem]{Corollary}

\newtheorem{claim}{Claim}
\newtheorem{fact}{Fact}
\newtheorem{proposition}{Proposition}[section]

\begin{document}
\title{\textbf{Some multivariable Rado numbers\footnote{Supported by the National Science Foundation of China
        (Nos. 12061059, 11601254, 11551001, 11161037, and 11461054) and the Qinghai Key
Laboratory of Internet of Things Project (2017-ZJ-Y21).}}}

\author{Gang Yang\footnote{College of Science, University of
Shanghai for Science and Technology, Shanghai 200093, China. {\tt
gangyang98@outlook.com; changxiang-he@163.com}}, \ \ Yaping
Mao\footnote{School of Mathematics and Statistics, Qinghai Normal
University, Xining, Qinghai 810008, China. {\tt
maoyaping@ymail.com}} \footnote{Corresponding author}
\footnote{Academy of Plateau Science and Sustainability, Xining,
Qinghai 810008, China}, \ \ Changxiang He \footnotemark[2], \ \ Zhao
Wang\footnote{College of Science, China Jiliang University, Hangzhou
310018, China. {\tt wangzhao@mail.bnu.edu.cn}}}

\date{}
\maketitle

\begin{abstract}
The Rado number of an equation is a Ramsey-theoretic quantity
associated to the equation. Let $\mathcal{E}$ be a linear equation.
Denote by $\operatorname{R}_r(\mathcal{E})$ the minimal integer, if
it exists, such that any $r$-coloring of $[1,
\operatorname{R}_r(\mathcal{E})]$ must admit a monochromatic
solution to $\mathcal{E}$. In this paper, we give upper and lower
bounds for the Rado number of $\sum_{i=1}^{m-2}x_i+kx_{m-1}=\ell
x_{m}$, and some exact values are also given. Furthermore, we derive
some results for the cases that $\ell=m=4$ and $m=5, \ell=k+i \
(1\leq i\leq 5)$. As a generalization, the \emph{$r$-color Rado
numbers} for linear equations
$\mathcal{E}_1,\mathcal{E}_2,...,\mathcal{E}_r$ is defined as the
minimal integer, if it exists, such that any $r$-coloring of $[1,
\operatorname{R}_r(\mathcal{E}_1,\mathcal{E}_2,\ldots,\mathcal{E}_r)]$
must admit a monochromatic solution to some $\mathcal{E}_i$, where
$1\leq i\leq r$. A lower bound for
$\operatorname{R}_r(\mathcal{E}_1,\mathcal{E}_2,\ldots,\mathcal{E}_r)$
and the exact values of $\operatorname{R}_2(x+y=z,\ell x=y)$ and $\operatorname{R}_2(x+y=z, x+a=y)$ were given by Lov\'{a}sz Local Lemma.\\[2mm]
\end{abstract}

\section{Introduction}\label{s:1}

One important result in Ramsey theory is the theorem of van der
Waerden \cite{Waerden27}, stating that for any positive integer $k$,
there exists an integer $w(r,k)$ such that any $r$-coloring of
$\{1,2,\ldots,w(r,k)\}$ yields a monochromatic $k$-term arithmetic
progression. This topic has received great attention and wide
research on the arithmetic progression, monochromatic arithmetic
progression, and their generalizations; see the book \cite{LaRo03}
and some papers \cite{BrLa99, Landman98, LaWy97, Robertson16,
Thanatipanonda09, Vijay12} for details.

Let $\mathbb{N}$ represent the set of natural numbers and let
$[a,b]$ denote the set $\{n\in \mathbb{N}\,|\,a\leq n\leq b\}$. A
function $\chi: [1,n]\longrightarrow [0, r-1]$ is referred to as a
\emph{$r$-coloring} of the set $[1,n]$. Given a $r$-coloring $\chi$
and a linear equation $\mathcal{E}$, a solution $(x_1, x_2, \ldots
,x_m)$ to the linear equation $\mathcal{E}$ is \emph{monochromatic}
if and only if $\chi(x_1)=\chi(x_2)=\cdots= \chi(x_m)$.

In the sequel, we use symbols $x,y,z,w$ with or without subscripts
to indicate variables in our equations. Unspecified (integer)
coefficients will be denoted by $a,b,c,d,k,\ell$, with or without
subscripts. We assume coefficients are positive integers, and so $ax+by+cz=dw$ is not the same as $ax+by=cz+dw$.

In 1933, Rado \cite{Rodo33} generalized the concept of Schur numbers
to arbitrary systems of linear equations. The Rado number of an
equation is a Ramsey-theoretic quantity associated to the
equation.

\begin{definition}\label{defi-1}
The \emph{$r$-color Rado number} for a linear equation
$\mathcal{E}$ is defined as the minimal integer, if it exists, such
that any $r$-coloring of $[1, \operatorname{R}_r(\mathcal{E})]$ must
admit a monochromatic solution to $\mathcal{E}$.
\end{definition}

For $r=2$ we may write simply $\operatorname{RR}(\mathcal{E})$.

Harborth and Maasberg \cite{HaMa97} studied the $2$-color Rado
numbers of the equation $a(x+y)=bz$. Kosek and Schaal \cite{KoSc01}
investigated the $2$-color Rado numbers of the equation
$\sum^{m-1}_{i=1} x_i+c=x_m$ for negative values of $c$. Hopkins and
Schaal \cite{HoSc05} studied the Rado numbers for
$\sum^{m-1}_{i=1}a_ix_i=x_m$. Adhikari et al. \cite{ABEMR16} and
Boza et al. \cite{BMRS28} got the results for the $k$-color Rado
number for the equation $x_1+x_2+\cdots+x_n+c=x_{n+1}$. For more
details on the Rado and Schur numbers, we refer to the book
\cite{LaRo03} and papers \cite{ABEMR16, BMRS28, HaMa97, HoSc05,
JoSc05, KoSc03, KoSc01, KSW09, MyRo07, Rodo33, RoMy08, SaWy08,
YaXi15}.

It is natural to have the following generalization of Rado numbers.
\begin{definition}\label{defi-1}
The \emph{$r$-color Rado number}
for linear equations $\mathcal{E}_1,\mathcal{E}_2,\ldots ,\mathcal{E}_r$
is defined as the minimal integer, if it exists, such that any
$r$-coloring of $[1,
\operatorname{R}_r(\mathcal{E}_1,\mathcal{E}_2,\ldots,\mathcal{E}_r)]$
must admit a monochromatic solution to some $\mathcal{E}_i$, where
$1\leq i\leq r$.
\end{definition}

If the above $r$-color Rado number does not exist, then we write
$\operatorname{R}_r(\mathcal{E}_1,\mathcal{E}_2,...,\mathcal{E}_r)=\infty$.
If $r=2$, then we write
$\operatorname{RR}(\mathcal{E}_1,\mathcal{E}_2)$ for short.

The following probability result, due to L. Lov\'{o}sz
\cite{ErdosLovasz}, fundamentally improves the existence argument in
many instances. Let $A_1,\ldots,A_n$ be events in a probability
space $\Omega$. We say that the graph $\Gamma$ with vertex set
$\{1,2,\ldots,n\}$ is a \emph{dependency graph}
$\{A_1,A_2,\ldots,A_n\}$ if:
\begin{center}
$\{i\}$ not joined to $j_1,j_2,\ldots,j_s$ $\Longrightarrow $ $A_i$
and $A_{j_1}\cap A_{j_2}\cap \cdots \cap A_{j_s}$ are independent.
\end{center}

\begin{theorem} {\upshape (\textbf{Lov\'{a}sz Local Lemma} \cite{ErdosLovasz})}\label{th-LLL}
Let $A_1,\ldots,A_n$ be events in a probability space $\Omega$ with
dependence graph $\Gamma$. Suppose that there exists
$x_1,\ldots,x_n$ with $0<x_i\leq 1$ such that
$$
\Pr[A_i]<(1-x_i)\prod_{\{i,j\}\in \Gamma}x_j, \ 1\leq i\leq n.
$$
Then $\Pr[\bigwedge_{i} \overline{A_i}]>0$.
\end{theorem}

A slightly more convenient form of the local lemma results from the
following observation. Set
$$
y_i=\frac{1-x_i}{x\Pr[A_i]},
$$
so that
$$
x_i=\frac{1}{1+y_i\Pr[A_i]}.
$$

Since $1+z\leq \exp(z)$, we have:
\begin{corollary}[\cite{GrahamRothschildSpencer}]\label{Lemma-LLL}
Suppose that $A_1,\ldots,A_n$ are events in a probability space
having dependence graph $\Omega$, and there exist positive
$y_1,y_2,\ldots,y_n$ satisfying
$$
\ln y_i>\sum_{\{i,j\}\in \Gamma}y_j\Pr[A_j]+y_i\Pr[A_i],
$$
for $1\leq i\leq n$. Then $\Pr[\bigwedge \overline{A_i}]>0$.
\end{corollary}

By Corollary \ref{Lemma-LLL}, we can derive the following lower
bound in Section $2$.

\begin{theorem}\label{th-Lower-LLL}
Let $m_1,\ldots,m_r$ be the number of variables of equation
$\mathcal{E}_i$ with $m_1\geq 3$ and $m_1\leq \cdots \leq m_r$. For
all large $m_r$,
\begin{equation}          \label{eqlll}
\operatorname{R}_r(\mathcal{E}_1,\mathcal{E}_2,...,\mathcal{E}_r)\geq
\frac{m_r(r-1)}{c_2(\ln m_r(r-1)-\ln c_2)},
\end{equation}
where $c_1,c_2,c_3$ are positive such that $c_3-c_1^{m_1}c_2>0$ and
$c_3-c_1c_2+m_r<0$.
\end{theorem}

\begin{theorem}\label{th-twoo}
$(i)$ For $k\geq 3$, $\operatorname{RR}(x+y=z,\ell x=y)=5k$.

$(ii)$ For $a\geq 2$, we have
$$
\operatorname{RR}(x+y=z, x+a=y)=
\begin{cases}
\infty &\text{$a$ is odd,}\\
3a+2 &\text{$a$ is even.}\\
\end{cases}
$$
\end{theorem}
\begin{itemize}
\item[]
Let $\mathcal{E}(m,k,\ell)$ denote the equation
\begin{equation}          \label{eq01}
\sum_{i=1}^{m-2}x_i+kx_{m-1}=\ell x_{m}.
\end{equation}
\item[]
Let $\mathcal{E}_1(n,\ell)$ denote the equation
\begin{equation}          \label{eq1}
\sum_{i=1}^{n-1}x_i=\ell x_{n}.
\end{equation}

\item[]
Let $\mathcal{E}_2(\sum_{i=1}^{m-1}a_i,\ell)$ denote the equation
\begin{equation}          \label{eq2}
\sum_{i=1}^{m-1}a_ix_i=y_1+y_2+\cdots+y_{\ell}.
\end{equation}
\end{itemize}

Let $\operatorname{R}_r(\mathcal{E}(m,k,\ell))$,
$\operatorname{R}_r(\mathcal{E}_1(n,\ell))$,
$\operatorname{R}_r(\mathcal{E}_2(\sum_{i=1}^{m-1}a_i,\ell))$ be the
$r$-color Rado number of Equations (\ref{eq01}), (\ref{eq1}),
(\ref{eq2}) respectively. If $r=2$, then
$\operatorname{RR}(\mathcal{E}(m,k,\ell))=\operatorname{R}_2(\mathcal{E}(m,k,\ell))$,
$\operatorname{RR}(\mathcal{E}_1(n,\ell))=\operatorname{R}_2(\mathcal{E}_1(n,\ell))$,
and
$\operatorname{RR}(\mathcal{E}_2(\sum_{i=1}^{m-1}a_i,\ell))=\operatorname{R}_2(\mathcal{E}_2(\sum_{i=1}^{m-1}a_i,\ell))$.

\begin{remark}
$(i)$ If $y_1=y_2=\cdots=y_{\ell}=x_m$ and $a_i=1$ for
$i=1,2,...,m-2$ and $a_{m-1}=k$, then the solutions of (\ref{eq2})
equal to the solutions of (\ref{eq01}).

$(ii)$ If $a_i=1$ for $i=1,2,\ldots,m-2$ and $a_{m-1}=k$, then the
solutions of (\ref{eq01}) are also the solutions of (\ref{eq2}).
Furthermore, we have $\operatorname{RR}(\mathcal{E}_2(k+m-2,\ell))\leq
\operatorname{RR}(\mathcal{E}(m,k,\ell))$.
\end{remark}

Saracino \cite{SaD13} studied the $2$-color Rado number of $\mathcal{E}_1(n,\ell)$.

\begin{theorem}[\cite{SaD13}]\label{th-SaD13-1}
For every integer $\ell\geq 3$ and $n\geq \ell/2+1$, we have {\small
\begin{table}[h]
\caption{$\operatorname{RR}(\mathcal{E}_1(n,\ell))$\,.}
\label{tab:1} \centering
\begin{tabular}{ccc}
\cline{1-2} \hline Values  & Conditions on $n,\ell$ \\[0.1cm]
\hline
$1$ & $n=\ell+1$; \\[0.1cm]
$3$ & $\ell+2 \leq n \leq \frac{3 \ell}{2}+1$ and $\ell \equiv n-1(\bmod\, 2)$; \\[0.1cm]
  & $\frac{2 \ell}{3}+1 \leq n \leq \ell$ and $\ell \geq 4$ and $\ell \equiv
n-1(\bmod\, 2)$;\\[0.1cm]
$4$ & $\ell+2 \leq n \leq \frac{3 \ell}{2}+1$ and $\ell \not \equiv n-1(\bmod\, 2)$;\\[0.1cm]
    & $\frac{3 \ell}{2}+1 < n \leq 2 \ell+1$ and $\ell \equiv n-1(\bmod\, 3)$;\\[0.1cm]
    & $\frac{2 \ell}{3}+1 \leq n \leq \ell$ and $\ell \geq 4$ and $\ell \not \equiv n-1(\bmod\, 2)$;\\[0.1cm]
   & $\frac{\ell}{2}+1 \leq n<\frac{2 \ell}{3}+1$ and $\ell \geq 4$ and $\ell
 \equiv n-1(\bmod\, 3)$;\\[0.1cm]
$5$ & $\frac{3 \ell}{2}+1 < n \leq 2 \ell+1$ and $\ell\not \equiv n-1(\bmod\, 3)$; \\[0.1cm]
   & $\frac{\ell}{2}+1 \leq n<\frac{2 \ell}{3}+1$ and $\ell \geq 4$ and $\ell \not \equiv n-1(\bmod\, 3)$\\
  & except $\ell=4,5$, $n=\ell-1$ and $10\leq \ell\leq 14, n=\ell-4$;\\[0.1cm]
$6$ & $10 \leq \ell \leq 14$ and $n=\ell-4$;\\[0.1cm]
$9$  & $\ell=n=3$, or $\ell=5$ and $n=4$;\\[0.1cm]
$10$  & $\ell=4$ and $n=3$;\\[0.1cm]
$\left\lceil\frac{n-1}{\ell}\left\lceil\frac{n-1}{\ell}\right\rceil\right\rceil$
&  $n \geq 2 \ell+2$.\\[0.1cm]
\cline{1-2}
\end{tabular}
\end{table}
}
\end{theorem}
Saracino \cite{Saracino16} obtained some good results for $2$-color
Rado number of the Equation (\ref{eq2}).
\begin{theorem}[\cite{Saracino16}]\label{th-Saracino-LBounds01}
Let $\ell\geq 2$ and $a_1,\ldots,a_{\ell}$ be positive integers. If $\sum_{i=1}^{m-1}a_i\leq 2\ell$, then
$$
\operatorname{RR}\left(\mathcal{E}_2\left(\sum_{i=1}^{m-1}a_i,\ell\right)\right)= \operatorname{RR}\left(\mathcal{E}_1\left(\ell+1,\sum_{i=1}^{m-1}a_i\right)\right),
$$
where the pair $(\sum_{i=1}^{m-1}a_i,\ell)$ is none of $(3,2),
(4,2), (5,3), (10,5), (11,6), (12,7), (13,8), (14,9)$.
\end{theorem}

\begin{theorem}[\cite{Saracino16}]\label{th-Saracino-LBounds}
If $\ell\geq 2$ and $\sum_{i=1}^{m-1}a_i>2\ell$, then
$$
\operatorname{RR}\left(\mathcal{E}_2(m-1,\ell)\right)\geq \left\lceil
\frac{\sum_{i=1}^{m-1}a_i}{\ell}\left\lceil
\frac{\sum_{i=1}^{m-1}a_i}{\ell}\right\rceil\right\rceil.
$$
\end{theorem}

If $a_1=a_2=\cdots=a_{m-2}=1$, $a_{m-1}=k$ and
$y_1=y_2=\cdots=y_{\ell}=x_m$, we can stand up the relation for Equations (\ref{eq01}), (\ref{eq1}), (\ref{eq2}).
\begin{theorem}\label{th-UBounds01}
If $(k+m-2)/2 \leq \ell \leq k+2(m-2)$ and $\ell> k$, then
$$
\operatorname{RR}(\mathcal{E}_1(\ell+1,k+m-2))\leq
\operatorname{RR}\left(\mathcal{E}(m,k,\ell)\right)\leq \operatorname{RR}(\mathcal{E}_1(m-1,\ell-k)).
$$
\end{theorem}

\begin{theorem}\label{th-UBounds}
If $k<\ell<(k+m-2)/2$, then
$$
\left\lceil \frac{k+m-2}{\ell}\left\lceil
\frac{k+m-2}{\ell}\right\rceil\right\rceil\leq
\operatorname{RR}\left(\mathcal{E}(m,k,\ell)\right)\leq \operatorname{RR}(\mathcal{E}_1(m-1,\ell-k)).
$$
Furthermore, if $k(\ell-k)^2\geq (m-2)^3$ and $m\geq
(\ell-k)^2-(\ell-k)+2$, then
$$
\operatorname{RR}(\mathcal{E}(m,k,\ell))=
\left\lceil\frac{m-2}{\ell-k}\left\lceil\frac{m-2}{\ell-k}\right\rceil\right\rceil.
$$
\end{theorem}

For some special cases, we can get the exact values in Section $2$.
\begin{theorem}\label{th-Exact}
Let the pair $(k+m-2,\ell)$ be none of $(3,2),
(4,2), (5,3), (10,5), (11,6),(12,7), (13,8),$ $(14,9)$ and $\ell\geq 2$.

$(i)$ If $2(m-2)/3+k \leq \ell \leq 3(m-2)/2+k$ and $\ell \neq m+k-2$, then
$$
\operatorname{RR}\left(\mathcal{E}(m,k,\ell)\right)=
\begin{cases}
3 & \text{$k+m-2\equiv \ell \ (\bmod\, 2)$ and $\ell-k\equiv m-2 \ (\bmod\, 2)$},\\
4 & \text{$k+m-2\not \equiv \ell \ (\bmod\, 2)$ and $\ell-k\not
\equiv m-2 \ (\bmod \,2)$}.
\end{cases}
$$

$(ii)$ If $\ell=k+m-2$, then
$\operatorname{RR}(\mathcal{E}(m,k,\ell))=1$.

$(iii)$ If $(m-2)/2+k \leq \ell < 2(k+m-2)/3$ or $3(m-2)/2+k< \ell\leq k+2m-4$ and $m\geq k$, then
$$
\operatorname{RR}\left(\mathcal{E}(m,k,\ell)\right)=
\begin{cases}
4 & \text{$k+m-2\equiv \ell \ (\bmod \,3)$ and $\ell-k\equiv m-2 \ (\bmod \,3)$},\\
5 & \text{$k+m-2\not \equiv \ell \ (\bmod\, 3)$ and $\ell-k\not
\equiv m-2 \ (\bmod\, 3)$}.
\end{cases}
$$

$(iv)$ If $k\geq m-2$, then
$\operatorname{RR}\left(\mathcal{E}(m,(m-2)k,(m-2)k)\right)=2k$.
\end{theorem}

Robertson and Myers \cite{RoMy08} studied the case that $m=4$ of
Equation (\ref{eq01}) and proposed the following conjecture.

\begin{conjecture}[\cite{RoMy08}]\label{conj}
For $\ell\geq 2$ fixed and $k\geq \ell+2$, we have
$$
\operatorname{RR}(x+y+kz=\ell w)=\left\lfloor
\frac{k+\ell+1}{\ell}\right\rfloor^2+O\left(\frac{k}{\ell^2}\right).
$$
\end{conjecture}

Saracino and Wynne \cite{SaWy08} showed that the conjecture is
partially not true for $\ell=3$, and they obtained the exact value
of $2$-color Rado number of $x+y+kz=3w$.

From \cite{RoMy08}, the values of $\operatorname{RR}(x+y+kz=4w)$ for
$k=1,2,3,4,5,16$ are $4,1,4,4,6,24$, respectively. The following
theorem shows the values of $\operatorname{RR}(x+y+kz=4w)$ for other
positive integer $k$, and the proofs are given in Section $3$.

\begin{theorem}\label{th-Exact-1}
For $k\geq 6$ and $k\neq 16$, we have
$$
\operatorname{RR}(x+y+kz=4w)=
\begin{cases}
\frac{k^2+6k+16}{16} & \text{for $k\equiv 0 \pmod {16}$},\\
\frac{k^2+5k+10}{16} & \text{for $k\equiv 1,10\pmod {16}$},\\
\frac{k^2+4k+4}{16} & \text{for $k\equiv 2 \pmod {16}$},\\
\frac{k^2+8k+7}{16} & \text{for $k\equiv 7 \pmod {16}$},\\
\frac{k^2+7k+8}{16} & \text{for $k\equiv 8 \pmod {16}$},\\
\frac{k^2+6k+9}{16} & \text{for $k\equiv 9 \pmod {16}$},\\
\frac{k^2+9k+20}{16} & \text{for $k\equiv 11 \pmod {16}$},\\
\frac{k^2+8k+16}{16} & \text{for $k\equiv 12 \pmod {16}$},\\
\frac{k^2+7k+12}{16} & \text{for $k\equiv 13 \pmod {16}$},\\
\frac{k^2+6k+8}{16} & \text{for $k\equiv 14 \pmod {16}$},\\
\frac{k^2+7k+22}{16} & \text{for $k\equiv 15 \pmod {16}$},\\
\end{cases}
$$
and
$$
\operatorname{RR}(x+y+kz=4w)=
\begin{cases}
\frac{4k^2+30k+66}{64} & \text{for $k\equiv 3,35 \pmod {64}$},\\
\frac{4k^2+30k+34}{64} & \text{for $k\equiv 19,51 \pmod {64}$},\\
\frac{4k^2+26k+24}{64} & \text{for $k\equiv 4,36 \pmod {64}$},\\
\frac{4k^2+26k+56}{64} & \text{for $k\equiv 20,52 \pmod {64}$},\\
\frac{4k^2+22k+46}{64} & \text{for $k\equiv 5,37 \pmod {64}$},\\
\frac{4k^2+22k+78}{64} & \text{for $k\equiv 21,53 \pmod {64}$},\\
\frac{4k^2+18k+68}{64} & \text{for $k\equiv 6,38 \pmod {64}$},\\
\frac{4k^2+18k+36}{64} & \text{for $k\equiv 22,54 \pmod {64}$}.\\
\end{cases}
$$
\end{theorem}

In terms of the current results in \cite{RoMy08} and this paper, we
modify Conjecture \ref{conj} and propose the following conjecture.
\begin{conjecture}
For $\ell\geq 2$ fixed and $k\geq \ell+2$, we have
$$
\operatorname{RR}(x+y+kz=\ell w)=\left\lfloor
\frac{k+\ell+1}{\ell}\right\rfloor^2+O\left(\frac{k}{\ell^3}\right).
$$
\end{conjecture}

For five variable Rado numbers, the exact values for
$\operatorname{RR}(\mathcal{E}(5,k,k+j))$ for $1\leq j\leq 5$ are
determined in Section $4$.

\section{Results for multivariable Rado numbers}

At first, we give the proof of Theorem
\ref{th-Lower-LLL}.\vskip0.2mm

\noindent \textbf{Proof of Theorem \ref{th-Lower-LLL}}: Color each
integer of $[N]$ by colors $c_1,\ldots,c_r$ randomly and
independently, in which each integer is colored $c_i$ with
probability $p_i$, where $\sum_{i=1}^rp_i=1$. For each solution
$S_i$ of $\mathcal{E}_i$ of $[N]$, let $A_{S_i}$ denote the event
that $S_i$ is a monochromatic solution with $c_i$, where $1\leq
i\leq r$. Then
$$
\Pr[A_S]=p_{i}^{m_i}.
$$

\begin{fact}\label{Fact-1}
The number of solutions of a equation $\mathcal{E}_i$ with $m_i$
variables containing $x$ in $[1,n]$ is at most $m_in^{m_i-2}$.
\end{fact}

Let $\Gamma$ denote the graph corresponding to all possible
$A_{S_i}$, where $\{A_{S_i},A_{S_j}\}$ is an edge of $\Gamma$ if and
only if $|S_i\cap T_j|\geq 1$ (i.e., the events $A_{S_i}$ and
$A_{S_j}$ are dependent), the same applies to pairs of the form
$\{A_{S_i},A_{S_i'}\}$ and $\{A_{S_j},A_{S_j'}\}$. Let $N_{A_iA_i}$
denote the number of vertices of the form $A_{S_i}$ for some $S_i$
joined to some other vertex of this form, and let $N_{A_iA_j} $ and
$N_{A_jA_i}$ be defined analogously.

\begin{fact}\label{Fact-2}
$N_{A_iA_i}\leq m_i^2n^{m_i-2}$.
\end{fact}

Similarly, we have
$$
N_{A_iA_j}\leq m_im_jn^{m_j-2}, \ N_{A_jA_i}\leq m_im_jn^{m_i-2}.
$$

If there exist positive $p_i,y_i \ (1\leq i\leq k)$ such that
$$
\log y_i > y_i\cdot p_i^{m_i}\cdot m_i^2n^{m_i-2}+\sum_{j=1, \ j\neq
i}^ry_j\cdot p_j^{m_j}\cdot m_im_jn^{m_j-2}
$$
then
$\operatorname{GR}_r(\mathcal{E}_1,\mathcal{E}_2,...,\mathcal{E}_r)>n$.

We shall choose $N,p_i,y_i$ in order by
$$
p_1=\cdots=p_{r-1}=\frac{p}{r-1}, \ \ p_r=1-p, \ \ m_r=c_2n(\ln
n)(r-1)^{-1},
$$
and
$$
p=c_1n^{-1}(r-1), \ \ y_1=\cdots=y_{r-1}=1+\epsilon, \ \
y_r=e^{{c_3}\ln n}.
$$

Observe that for $1\leq i\leq r-1$, we have
\begin{eqnarray*}
\log y_i &>&y_1\cdot p_1^{m_1}\cdot (r-1)\cdot
m_im_{r-1}n^{m_{r-1}-2}+(1-p)^{m_r}\cdot
m_im_rn^{m_r-2}y_r\\
&>&y_i\cdot p_i^{m_i}\cdot m_i^2n^{m_i-2}+\sum_{j=1, \ j\neq
i}^rp_j^{m_j}\cdot m_im_jn^{m_j-2}y_j
\end{eqnarray*}
and
\begin{eqnarray*}
\log y_r &>&y_1\cdot p_1^{m_1}\cdot (r-1)\cdot
m_rm_{r-1}n^{m_{r-1}-2}+(1-p)^{m_r}\cdot
m_r^2n^{m_r-2}y_r\\
&>&\sum_{j=1, \ j\neq r}^ry_j\cdot p_j^{m_j}\cdot
m_rm_jn^{m_j-2}+y_r\cdot p_r^{m_r}\cdot m_r^2n^{m_r-2}.
\end{eqnarray*}

If $c_3-c_1^{m_1}c_2>0$, $c_3-c_1c_2+m_r<0$, and $m_r$ is
sufficiently large, then the inequations hold. Since
$$
m_r=c_2n(\ln n)(r-1)^{-1}\leq c_2n\left(\ln
m_r(r-1)c_2^{-1}\right)(r-1)^{-1},
$$
it follows that
$$
\operatorname{GR}_r(\mathcal{E}_1,\mathcal{E}_2,...,\mathcal{E}_r)\geq
\frac{m_r(r-1)}{c_2(\ln m_r(r-1)-\ln c_2)},
$$
and hence the result follows.\\

\begin{remark}
The lower bound of (\ref{eqlll}) holds under the condition
$c_3-c_1^{m_1}c_2>0$. Here, $c_1,c_2,c_3$ rely on the value of
$m_r$, and can be seen as the function of $m_r$. To improve the
lower bound for given specific equations, one need to improve Fact
\ref{Fact-1} first.
\end{remark}

Next, we give the proof of Theorem \ref{th-twoo} by the following two
lemmas.
\begin{lemma}
For $k\geq 3$, we have $\operatorname{RR}(x+y=z, kx=y)=5k$.
\end{lemma}
\begin{proof}
To show the lower bound, we color $1,4,2k,4k$ red and other integers
in $[1,5k-1]$ blue. It is clear that there is no red solution to
$x+y=z$. If $x\in [1,5k-1]\setminus \{1,4,2k,4k\}$, then $y=2k, 3k$
or $y\geq 5k$, and hence there is no blue solution to $kx=y$.

For the upper bound, let $R$ be the set of red integers, and $B$ be
the set of blue integers. For any $2$-coloring of $[1,5k]$, we
assume that there is neither red solutions to $x+y=z$ nor blue
solutions to $kx=y$. Suppose that $1\in R$. To avoid a red solution,
$2\in B$, and hence $2k\in R$. It implies that $4k \in B$ by
considering the solution $2k+2k=4k$. Then, we have $4\in R$ and so
$3,5\in B$. It shows that $3k,5k\in R$. However $2k+3k=5k$ is a red
solution, a contradiction. Suppose that $1\in B$. To avoid a blue
solution, $k\in R$, and hence $2k\in B$. It implies that $2k\in R$.
Then, we have $4\in B$, and so $4k\in R$. It shows that $3k,5k\in
B$. However $2k+3k=5k$ is a blue solution, a contradiction.
\end{proof}

\begin{lemma}\label{th-two}
For $a\geq 2$, we have
$$
\operatorname{RR}(x+y=z, x+a=y)=
\begin{cases}
\infty &\text{$a$ is odd,}\\
3a+2 &\text{$a$ is even.}\\
\end{cases}
$$
\end{lemma}
\begin{proof}
Suppose that $a$ is odd. We color the odd integers in $\mathbb{Z}^+$
red and the even integers in $\mathbb{Z}^+$ blue. If $x,y,z$ are
red, then $x+y$ must be even and hence there is no red solutions to
$x+y=z$. If $x,y$ are blue, then $x+a$ must be odd, and so there is
no blue solutions to $x+a=y$. It follows that
$\operatorname{RR}(x+y=z, x+a=y)=\infty$.

Suppose that $a$ is even. For the lower bound, we color the integers
in $[a+1,2a]$ red and the integers in $[1,a]\cup[2a+2,3a+1]$. If
$x,y \in [a+1,2a]$, then $x+y\in [2a+2,4a]$, and hence there is no
red solution to $x+y=z$. If $x\in [1,a]\cup[2a+2,3a+1]$, then
$x+a\in [a+1,2a]\cup[3a+2,4a+1]$, and hence there is no blue
solution to $x+a=y$. We now show the upper bound. For any
$2$-coloring of $[1,3a+2]$, we assume that there is neither red
solutions to $x+y=z$ nor blue solutions to $x+a=y$. Suppose that
$1\in R$. To avoid a red solution, $2\in B$, and hence $a+2\in R$.
Then we have $\frac{a+2}{2}\in B$. It follows that
$\frac{3a+2}{2}\in R$. We have $2a+4 \in B$ by considering
$a+2+a+2=2a+4$, then $a+4\in R$ and hence $\frac{3a+4}{2}\in R$. It
follows that $1+\frac{3a+2}{2}=\frac{3a+4}{2}$, a contradiction.
Suppose that $1\in B$. To avoid a blue solution, $a+1\in R$. If
$2\in R$, then $a+3, 4\in B$ and hence $a+4,3\in R$. It follows that
$a+1+3=a+4$, a contradiction. If $2\in B$, then $a+2\in R$, and
hence $\frac{a+2}{2}\in B$. This means that $\frac{3a+2}{2}\in R$
and hence $3a+2\in B$. Since $a+1\in R$, we have $2a+2\in B$, and
hence $2a+2+a=3a+2$, a contradiction.
\end{proof}

\noindent \textbf{Proof of Theorem \ref{th-UBounds01}}:
The $2$-color Rado number of $\mathcal{E}(m,k,\ell)$ is at least that of
\begin{equation}          \label{eq5}
x_1+\cdots+x_{m-2}+kx_{m-1}=y_1+\cdots+y_{\ell}
\end{equation}
and at most that of
\begin{equation}\label{eq6}
x_1+\cdots+x_{m-2}=(\ell-k)x_{m-1}.
\end{equation}
Let $\mathcal{E}_2(k+m-2,\ell)$ and $\mathcal{E}_1(m-2,\ell-k)$
denote the equations (\ref{eq5}) and (\ref{eq6}), respectively.
Since any monochromatic solution of $\mathcal{E}(m,k,\ell)$ is also a
monochromatic solution of $\mathcal{E}_2(k+m-2,\ell)$ and any monochromatic solution of $\mathcal{E}_1(m-2,\ell-k)$ is also a
monochromatic solution of $\mathcal{E}(m,k,\ell)$, it follows that
$$
\operatorname{RR}\left(\mathcal{E}_2(k+m-2,\ell)\right)\leq
\operatorname{RR}\left(\mathcal{E}(m,k,\ell)\right)\leq
\operatorname{RR}\left(\mathcal{E}_1(m-2,\ell-k)\right).
$$
From Theorem \ref{th-Saracino-LBounds01}, we have $\operatorname{RR}\left(\mathcal{E}_2(k+m-2,\ell)\right)=
\operatorname{RR}\left(\mathcal{E}_1(\ell+1,k+m-2)\right)$, where
$k+m-2\leq 2\ell$ and the pair $(k+m-2,\ell)$ is none of $(3,2),
(4,2), (5,3), (10,5), (11,6),$ $(12,7), (13,8), (14,9)$. Therefore, we have
$$
\operatorname{RR}\left(\mathcal{E}_1(\ell+1,k+m-2)\right)\leq
\operatorname{RR}\left(\mathcal{E}(m,k,\ell)\right)\leq
\operatorname{RR}\left(\mathcal{E}_1(m-2,\ell-k)\right).
$$

We now in a position to prove Theorem \ref{th-UBounds} by the
following two lemmas.

\begin{lemma}\label{lem-upBounds}
If $k+3\leq \ell \leq k+2(m-2)$, then
$$
\left\lceil \frac{k+m-2}{\ell}\left\lceil
\frac{k+m-2}{\ell}\right\rceil\right\rceil\leq
\operatorname{RR}\left(\mathcal{E}(m,k,\ell)\right)\leq \left\lceil
\frac{m-2}{\ell-k}\left\lceil
\frac{m-2}{\ell-k}\right\rceil\right\rceil.
$$
\end{lemma}
\begin{proof}
Since any solution to $\mathcal{E}_1(m-1,\ell-k)$ is also a solution
to $\mathcal{E}(m,k,\ell)$ for any $k\in \mathbb{Z}^+$, it follows
from Theorem \ref{th-SaD13-1} that
$\operatorname{RR}\left(\mathcal{E}(m,k,\ell)\right)\leq
\operatorname{RR}\left(\mathcal{E}_1(m-1,\ell-k)\right)$.

To show the lower bound, let $A=\lceil \frac{k+m-2}{\ell}\lceil
\frac{k+m-2}{\ell}\rceil\rceil$. We color all elements in $[1,
\lceil\frac{k+m-2}{\ell}\rceil-1]$ red and the remaining elements in
$[\lceil\frac{k+m-2}{\ell}\rceil,A-1]$ blue. If $x_1,
x_2,\ldots,x_{m-1}$ are all red, then $x_1+x_2+\ldots+x_{m-1}+kx_{m-1}\geq
k+m-2$, and hence $x_m\geq \lceil\frac{k+m-2}{\ell}\rceil$. Clearly,
there is no red solution. If $x_1, x_2,\ldots,x_{m-1}$ are all blue,
then $x_1+x_2+\ldots+x_{m-1}+kx_{m-1}\geq A$, and so $w\geq B$, and
hence there is no blue solution in
$\left[\left\lceil\frac{k+m-2}{\ell}\right\rceil, A-1\right]$.
\end{proof}

\begin{lemma}\label{le}
For all $k\geq \frac{(m-2)^3}{(\ell-k)^2}$ and $m\geq (\ell-k)^2-(\ell-k)+2$, we have
$$
\operatorname{RR}(\mathcal{E}(m,k,\ell))=
\left\lceil\frac{m-2}{\ell-k}\left\lceil\frac{m-2}{\ell-k}\right\rceil\right\rceil.
$$
\end{lemma}

\begin{proof}
For $m\geq (\ell-k)^2-(\ell-k)+2$, let
$B=\lceil\frac{m-2}{\ell-k}\lceil\frac{m-2}{\ell-k}\rceil\rceil$.
From Theorem \ref{th-SaD13-1}, we have
$\operatorname{RR}(\mathcal{E}(m,k,\ell))\leq
\operatorname{RR}(\mathcal{E}_1(m-1,\ell-k))=B$. It suffices to prove that
$$
\operatorname{RR}(\mathcal{E}(m,k,\ell))=\operatorname{RR}\left(\sum_{i=1}^{m-2}x_i+kx_{m-1}=\ell
x_{m}\right)\geq B.
$$
Suppose that there exists a $2$-coloring of $[1,B-1]$ that admit a
monochromatic solution to $\sum_{i=1}^{m-2}x_i+kx_{m-1}=\ell x_{m}$.
Then, we show that any solution to
$\sum_{i=1}^{m-2}x_i+kx_{m-1}=\ell x_{m}$ with $x_{i}\leq B-1$ must
have $x_{m-1}=x_{m}$, where $1\leq i\leq m$. If $x_{m-1}< x_{m}$,
then $\ell x_{m}\geq \ell (x_{m-1}+1)>kx_{m-1}+k$. Since $k\geq
\frac{(m-2)^3}{(\ell-k)^2}$, it follows that
$\sum_{i=1}^{m-2}x_i+kx_{m-1}\leq (m-2)(B-1)+kx_{m-1}\leq
kx_{m-1}+k< \ell x_m$, and hence there is no monochromatic solution.
If $x_{m-1}>x_m$, then $\ell x_m\leq kx_{m-1}-1)+(\ell-k)(B-1)$.
Since $\sum_{i=1}^{m-2}x_i+kx_{m-1}\geq m-2+kx_{m-1}$, we have
$m-2+kx_{m-1}\leq \sum_{i=1}^{m-2}x_i+kx_{m-1}=\ell x_m\leq
k(x_{m-1}-1)+(\ell-k)(B-1)$, and hence $k\leq (\ell-k)(B-1)+2-m$.
However, $k\geq \frac{(m-2)^3}{(\ell-k)^2}>(\ell-k)(B-1)+2-m$, where
$m\geq (\ell-k)^2-(\ell-k)+2$, a contradiction. Therefore, we have
$x_{m-1}=x_m$. This means that any solution to
$\sum_{i=1}^{m-2}x_i+kx_{m-1}=\ell x_{m}$ is also a solution to
$\sum_{i=1}^{m-2}x_i=(\ell-k)x_{m-1}$. From Dan Saracino's result,
there exists a $2$-coloring of $[1,B-1]$ with no monochromatic
solution to $\sum_{i=1}^{m-2}x_i=(\ell-k)x_{m-1}$, and hence there
exists a $2$-coloring of $[1,B-1]$ with no monochromatic solution to
$\sum_{i=1}^{m-2}x_i+kx_{m-1}=\ell x_{m}$, and so
$\operatorname{RR}\left(\mathcal{E}(m,k,\ell)\right)\geq B=
\left\lceil\frac{m-2}{\ell-k}\left\lceil\frac{m-2}{\ell-k}\right\rceil\right\rceil$.
\end{proof}

We now give the proof of Theorem \ref{th-Exact} by the following
four lemmas.
Let the pair $(k+m-2,\ell)$ be none of $(3,2),(4,2), (5,3), (10,5), (11,6),$ $(12,7), (13,8), (14,9)$ and $\ell\geq 2$.

\begin{lemma}\label{pro1}
Let $2(m-2)/3+k \leq \ell \leq 3(m-2)/2+k$ and $\ell \neq m+k-2$. Then
$$
\operatorname{RR}\left(\mathcal{E}(m,k,\ell)\right)=
\begin{cases}
3 & \text{$k+m-2\equiv \ell \ (\bmod\, 2)$ and $\ell-k\equiv m-2 \ (\bmod\, 2)$},\\
4 & \text{$k+m-2\not \equiv \ell \ (\bmod\, 2)$ and $\ell-k\not
\equiv m-2 \ (\bmod\, 2)$}.
\end{cases}
$$
\end{lemma}
\begin{proof}
If $2(m-2)/3+k \leq \ell \leq m+k-3$, then $(\ell-k)+2 \leq m-1 \leq
3(\ell-k)/2+1$ and $2(k+m-2)/3+k/3+1\leq \ell+1\leq k+m-2$.
If $m+k-1\leq \ell \leq 3(m-2)/2+k$, then $2(\ell-k)/3+1 \leq m-1 \leq
\ell-k$ and $(k+m-2)+2\leq \ell+1\leq 3(k+m-2)/2+1-k/2$.

From Theorem \ref{th-SaD13-1}, if $k+m-2\equiv \ell \ (\bmod\, 2)$
and $\ell-k\equiv m-2 \ (\bmod\, 2)$, then
$\operatorname{RR}(\mathcal{E}_1(\ell+1,k+m-2))=
\operatorname{RR}(\mathcal{E}_1(m-1,\ell-k))=3$. From Theorem
\ref{th-UBounds01}, we have
$\operatorname{RR}\left(\mathcal{E}(m,k,\ell)\right)=3$. If
$k+m-2\not \equiv \ell \ (\bmod\, 2)$ and $\ell-k\not \equiv m-2 \
(\bmod \,2)$, then $\operatorname{RR}(\mathcal{E}_1(\ell+1,k+m-2))=
\operatorname{RR}(\mathcal{E}_1(m-1,\ell-k))=4$. From Theorem
\ref{th-UBounds01}, we have
$\operatorname{RR}\left(\mathcal{E}(m,k,\ell)\right)=4$.
\end{proof}

\begin{lemma}
If $\ell=k+m-2$, then
$\operatorname{RR}(\mathcal{E}(m,k,\ell))=1$.
\end{lemma}

\begin{proof}
Clearly, $x_i=1$ for $1\leq i \leq m$ is the solution to the
$\sum_{i=1}^{m-2}x_i+kx_{m-1}=\ell x_{m}$. Then
$\operatorname{RR}(\mathcal{E}(m,k,\ell))=1$.
\end{proof}

\begin{lemma}
Suppose that $(m-2)/2+k \leq \ell < 2(k+m-2)/3$ or $3(m-2)/2+k< \ell\leq k+2m-4$ and $m\geq k$. Then
$$
\operatorname{RR}\left(\mathcal{E}(m,k,\ell)\right)=
\begin{cases}
4 & \text{$k+m-2\equiv \ell \ (\bmod\, 3)$ and $\ell-k\equiv m-2 \ (\bmod\, 3)$},\\
5 & \text{$k+m-2\not \equiv \ell \ (\bmod\, 3)$ and $\ell-k\not
\equiv m-2 \ (\bmod\, 3)$}.
\end{cases}
$$
\end{lemma}
\begin{proof}
If $(m-2)/2+k \leq \ell < 2(k+m-2)/3$, then $3(\ell-k)/2+1+k/2 \leq
m-1 \leq 2(\ell-k)+1$ and $(k+m-2)/2+1\leq \ell+1\leq 2(k+m-2)/3+1$.
If $3(k+m-2)/2+k < \ell \leq k+2m-4$, then $(\ell-k)/2+1 \leq m-1 <
2(\ell-k)/3+1$ and $3(k+m-2)/2+1+k < \ell+1 \leq 2(k+m-2)-k$. From
Theorem \ref{th-SaD13-1}, if $k+m-2\equiv \ell \ (\bmod\, 3)$ and
$\ell-k\equiv m-2 \ (\bmod\, 3)$, then
$\operatorname{RR}(\mathcal{E}_1(\ell+1,k+m-2))=
\operatorname{RR}(\mathcal{E}_1(m-1,\ell-k))=4$. From Theorem
\ref{th-UBounds01}, we have
$\operatorname{RR}\left(\mathcal{E}(m,k,\ell)\right)=4$. If
$k+m-2\not \equiv \ell \ (\bmod\, 3)$ and $\ell-k\not \equiv m-2 \
(\bmod\, 3)$, then $\operatorname{RR}(\mathcal{E}_1(\ell+1,k+m-2))=
\operatorname{RR}(\mathcal{E}_1(m-1,\ell-k))=5$. From Theorem
\ref{th-UBounds01}, we have
$\operatorname{RR}\left(\mathcal{E}(m,k,\ell)\right)=5$.
\end{proof}

\begin{lemma}
If $k\geq m-2$, then
$\operatorname{RR}\left(\mathcal{E}(m,(m-2)k,(m-2)k)\right)=2k$.
\end{lemma}
\begin{proof}
Consider the $2$-coloring of $[1,2k-1]$ defined by coloring the odd
integers red and the even integers blue. Suppose that there exists a
monochromatic solution to
$\sum_{i=1}^{m-2}x_i+(m-2)kx_{m-1}=(m-2)kx_{m}$. Then $(m-2)k\mid
\sum_{i=1}^{m-2}x_i$. Since $x_i \leq 2k-1$, where $1\leq i \leq
m-2$, it follows that $\sum_{i=1}^{m-2}x_i=(m-2)k$, and hence
$x_m=x_{m-1}+1$, which contradicts to the fact that there is no two
consecutive integers have same color. Thus, we have
$\operatorname{RR}\left(\mathcal{E}(m,(m-2)k,(m-2)k)\right)\geq 2k$.

To show that
$\operatorname{RR}\left(\mathcal{E}(m,(m-2)k,(m-2)k)\right)\leq 2k$,
we assume that there exists a $2$-coloring $\chi$ of $[1,2k]$ with
no monochromatic solution. Without loss of generality, we can assume
that $k$ is red. To avoid a red solutions $(k,k,...,k,k-1,k)$ and
$(k,k,...,k,k+1)$, $\chi(k-1)=\chi(k+1)$ is blue.
Furthermore, one can see that $\chi(2k)$ is red by the solution
$(2k,2k,...,2k,k-1,k+1)$. Therefore, using the solution
$(2k,2k,...,2k,k-2,k)$, $\chi(k-2)$ is blue. If $m$ is odd, then
$(k-1,...,k-1,k+1,...,k+1,k-2,k-2,k-1)$ is a blue solution, where
the number of $k-1$ (resp. $k+1$) is $\frac{m-3}{2}-1$ (resp.
$\frac{m-3}{2}+1$), a contradiction. If $m$ is even, then
$(k-1,...,k-1,k+1,...,k+1,k-2,k-1)$ is a blue solution, where the
number of $k-1$ (resp. $k+1$) is $\frac{m-2}{2}$ (resp.
$\frac{m-2}{2}$), a contradiction.
\end{proof}

\section{Results for four variable Rado numbers}

Let $R$ denote the set of red elements and $B$ the set of blue elements and $k\geq 32$.
\begin{lemma}\label{lem1}
Let $M$ be a positive integer which makes a $2$-coloring of $[1,M]$ admits no monochromatic solution to $x+y+kz=4w$. If $k> 4j-1, jk\leq M$ and $2\in R$, then $1, 2, 3,\ldots,4j-1\in R$ and $k, 2k,\ldots, jk \in B$.
\end{lemma}
\begin{proof}
Since $k>4j-1$, it follows that the sets $\{1, 2, 3,\ldots, 4j-1\}$ and $\{k, 2k,\ldots,
jk\}$ are disjoint. Suppose that a $2$-coloring satisfying the hypotheses of the lemma has been given. We assume $2\in R$, then $k\in B$ by considering $(k,k,2,k)$.
\begin{claim}\label{claim1}
$4\in R$.
\end{claim}
\begin{proof}
Assume, for a contradiction, that $4\in B$. Then $2k\in R$ by
considering $(2k, 2k, 4, 2k)$, and $k+2\in R$ comes from $(4, 4, 4,
k+2)$. From $(k+2, k+2, 2, k+1)$, we have $k+1\in R$, and hence
$(k+1, 3, 3, k+1)$ shows that $3\in R$. As a consequence, we see
that $6\in R$ by considering $(k+2, 6, 3, k+2)$. Since $(3k, 3k, 6,
3k)$, it follows that $3k\in R$. But $(3k, 3k, 2, 2k)$ is a red
solution, a contradiction.
\end{proof}

From Claim \ref{claim1}, we have $4\in R$. Furthermore, we have the
following claim.
\begin{claim}\label{claim2}
$3\in R$.
\end{claim}
\begin{proof}
Assume, for a contradiction, that $3\in B$. Then $k+1\in R$ by
considering $(3, k+1, 3, k+1)$, and $k+2\in B$ comes from $(k+2,
k+2, 2, k+1)$. From $(4, 2k, 2, k+1)$, we have $2k\in B$.
Furthermore, $(2k, k, 1, k)$ shows that $1\in R$. As a consequence,
we see that $3k\in B$ by considering $(4, 3k, 1, k+1)$. But $(3k,
2k, 3, 2k)$ is a blue solution, a contradiction.
\end{proof}
Since $4\in R$, we have $2k\in B$ by considering $(2k, 2k, 4, 2k)$. From
here, we use $(2k, k, 1, k)$ to see that $1\in R$. Now, suppose that
$j$ is a positive integer such that $jk\leq M$.
If $j=1$, we have $1,2,3\in R$ and $k\in B$. If $j=2,3,4$, by considering $(2k, 2k, 4i, (i+1)k), (k, 2k, 4i+1, (i+1)k), (k, k, 4i+2, (i+1)k), (2k, 3k, 4i+3, (i+2)k)$ and $(ik, ik, 2i, ik)$, where $1\leq i \leq j$, it is clear that $1, 2, 3,..., 4j-1\in R$ and $k, 2k,..., jk \in B$. For $j\geq 4$,
suppose that $1, 2, 3,..., 4(j-1)-1\in R$ and $k, 2k,..., (j-1)k \in B$, then $2j, 2(j+1)\in R$ and hence $4j-1\in R$ and $jk\in B$ by considering $(2k, 3k, 4j-1, (j+1)k)$ and $(jk, jk, 2j, jk)$. It follows that $1, 2, 3,..., 2j,...,4j-1\in R$ and $k, 2k,..., jk \in B$.
\end{proof}

We now give the proof of Theorem \ref{th-Exact-1} by the following
propositions. The proofs of some propositions are similarly to the
propositions here, and we move these proofs to Appendix I.

\begin{proposition}\label{pro3-1}
If $k\equiv 0 \pmod {16}$, then
$\operatorname{RR}(x+y+kz=4w)=\frac{k^2+6k+16}{16}$.
\end{proposition}
\begin{proof}
To show the upper bound, suppose that there is no monochromatic
solution for any $2$-coloring of $[1,\frac{k^2+6k+16}{16}]$. From
Lemma \ref{lem1}, $1, 2,...,\frac{k-4}{4}\in R$ and $k, 2k,\ldots,
\frac{k^2}{16}\in B$. Furthermore, $(2, 2, 1, \frac{k+4}{4})$ and
$(4, 4, 1, \frac{k+8}{4})$ show that $\frac{k+4}{4}, \frac{k+8}{4}\in B$.
We have $\frac{k^2+6k+16}{16}\in R$ by considering $(\frac{k+8}{4},
\frac{k+8}{4}, \frac{k+4}{4}, \frac{k^2+6k+16}{16})$. Since
$\frac{2k+48}{16}\in R$ and $\frac{3k+16}{16}\in R$, then
$(\frac{k^2+6k+16}{16},\frac{2k+48}{16},\frac{3k+16}{16},
\frac{k^2+6k+16}{16})$ is a red solution, a contradiction.

For the lower bound, $\frac{k^2+6k+16}{16}-1=\frac{k^2+6k}{16}$. Color the integers in
$[1,\frac{k}{4}]$ red and the integers in
$[\frac{k}{4}+1,\frac{k^2+6k}{16}]$ blue. Let $x,y,z,w$ be a
monochromatic solution of $x+y+kz=4w$. If $x, y, z$ are all red,
then $x+y+kz\geq k+2$, and hence $w\geq
\lceil\frac{k+2}{4}\rceil=\frac{k+4}{4}$, and so there is no red
solution. If $x, y, z$ are all blue, then $x+y+kz\geq
\frac{k^2+6k+8}{4}$, and hence $w\geq
\lceil\frac{k^2+6k+8}{16}\rceil=\frac{k^2+6k+16}{16}>
\frac{k^2+6k}{16}$, and so there is no blue solution.
\end{proof}

\begin{proposition}\label{pro3-2}
If $k\equiv 1 \pmod {16}$, then
$\operatorname{RR}(x+y+kz=4w)=\frac{k^2+5k+10}{16}$.
\end{proposition}

\begin{proposition}\label{pro3-3}
If $k\equiv 2 \pmod {16}$, then
$\operatorname{RR}(x+y+kz=4w)=\frac{k^2+4k+4}{16}$.
\end{proposition}

\begin{proposition}\label{pro3-4}
If $k\equiv 3,35 \pmod {64}$, then
$\operatorname{RR}(x+y+kz=4w)=\frac{4k^2+30k+66}{64}$.
\end{proposition}
\begin{proof}
If $k\geq 64$. To show the upper bound, suppose that there is no monochromatic
solution for any $2$-coloring of $[1,\frac{4k^2+30k+66}{64}]$. From
Lemma \ref{lem1}, we know that $1, 2,\ldots,\frac{k-3}{4}-1\in R$. The
solutions $(10, \frac{k-7}{4}, 1, \frac{5k+33}{16})$,
$(4,5,1,\frac{k+9}{4})$ and $(3, 2, 1, \frac{k+5}{4})$ show that
$\frac{5k+33}{16},\frac{k+9}{4},$
$\frac{k+5}{4}\in B$. By considering
$(\frac{5k+33}{16}, \frac{5k+33}{16}, \frac{k+5}{4},
\frac{4k^2+30k+66}{64})$, we have $\frac{4k^2+30k+66}{64}\in R$. We know that
$\frac{k^2+7k+18}{16}\in R$ by considering
$(\frac{k+9}{4},\frac{k+9}{4},\frac{k+5}{4}, \frac{k^2+7k+18}{16})$.
Since $3, \frac{3k+23}{16}\in R$, it follows that $(3,\frac{k^2+7k+18}{16}, \frac{3k+23}{16},$
$\frac{4k^2+30k+66}{64})$ is a red solution, a contradiction.

For the lower bound, $\frac{4k^2+30k+66}{64}-1=\frac{4k^2+30k+2}{64}$. Color $[1,\frac{k+1}{4}]\cup [\frac{k^2+7k+18}{16},\frac{4k^2+30k+2}{64}]$ red,
$[\frac{k+5}{4},\frac{k^2+7k+2}{16}]$ blue. Suppose that
there is a monochromatic solution of $x+y+kz=4w$, say $x,y,z,w$. For
any solution with $x,y,z$ all blue, we have $4w\geq
\frac{k+5}{4}(k+2)$, and so $w$ must be red. Thus, every
monochromatic solution is red. For every such solution, $w\geq
\frac{k^2+7k+18}{16}$. Recall that $z\leq \frac{k+1}{4}$.

\begin{claim}\label{claim3}
$x$ or $y$ is at least $\frac{k^2+7k+18}{16}$.
\end{claim}
\begin{proof}
Assume, to the contrary, that $x,y<\frac{k^2+7k+18}{16}$. Then $x, y
\leq \frac{k+1}{4}$, and hence $\frac{k^2+7k+18}{16} \leq w \leq
\lfloor\frac{k^2+3k+2}{16}\rfloor$, a contradiction.
\end{proof}

From Claim \ref{claim3}, without loss of generality, assume
that $x \geq \frac{k^2+7k+18}{16}$. It follows from $y+kz=4w-x$
that
$$
\frac{4(k^2+7k+18)}{16}-\frac{4k^2+30k+2}{64} \leq y+kz \leq
\frac{4(4k^2+30k+2)}{64}-\frac{k^2+7k+18}{16},
$$
i.e.,
$$
\frac{12k^2+82k+286}{64} \leq y+kz \leq \frac{12k^2+92k-64}{64}.
$$
If $z \geq \frac{3k+23}{16}$, then $y \leq
\frac{12k^2+92k-64}{64}-\frac{3k^2+23k}{16}=-1<0$, which is
impossible. If $z \leq \frac{3k+7}{16}$, then $y \geq
\frac{12k^2+82k+286}{64}-\frac{3k^2+7k}{16}=\frac{54k+286}{64}>\frac{k+1}{4}$.
Since $y \in R$, then $y\geq \frac{k^2+7k+18}{16}$. Assume
that
$x=\frac{4k^2+28k+a}{64},y=\frac{4k^2+28k+b}{64},w=\frac{4k^2+28k+c}{64}$,
with $a,b,c \in[72,2k+2]$. Therefore,
$$
z=\frac{4w-x-y}{k}=\frac{8k^2+56k+4c-a-b}{64k}=\frac{8k+56+(4c-a-b)/k}{16},
$$
Since $\frac{284-4k}{k} \leq \frac{4c-a-b}{k} \leq
\frac{8k-136}{k}$, it follows that $\frac{4c-a-b}{k}\in[-4,7]$, and
$z\in [\frac{8k+52}{64}, \frac{8k+63}{64}]$ is not an integer, and
so there are no monochromatic solutions.

If $32\leq k< 64$, then $k=35$. To show the upper bound,
suppose that there is no monochromatic solution for any $2$-coloring
of $[1,94]$. From Lemma \ref{lem1}, we have $1,2,\ldots,7\in R$. The
solutions $(2,3,1,10)$, $(4,5,1,11)$, $(13,4,1,13)$, $(24,2,2,24)$,
and $(82,1,7,82)$ show that $10,11,13,24,82\in B$. Then we have
$8,93,94\in R$ by considering $(24,24,8,82)$, $(11,11,10,93)$, and
$(13,13,10,94)$. Therefore, $(3,93,8,94)$ is a red solution, a
contradiction. For the lower bound, then $94-1=93$ and we color
$[1,9]\cup\{93\}$ red, $[10,92]$ blue. Clearly, there are no
monochromatic solutions.
\end{proof}

\begin{proposition}\label{pro3-5}
If $k\equiv 19,51 \pmod {64}$, then
$\operatorname{RR}(x+y+kz=4w)=\frac{4k^2+30k+34}{64}$.
\end{proposition}
\begin{proof}
To show the upper bound, suppose that there is no monochromatic
solution for any $2$-coloring of $[1,\frac{4k^2+30k+34}{64}]$. From
Lemma \ref{lem1}, we have $1, 2,...,\frac{k-3}{4}-1\in R$. The
solutions $(6, \frac{k-7}{4}, 1, \frac{5k+17}{16})$,
$(4,5,1,\frac{k+9}{4})$ and $(3, 2, 1, \frac{k+7}{4})$ show that
$\frac{5k+17}{16},\frac{k+9}{4}, \frac{k+5}{4}\in B$. We
have $\frac{4k^2+30k+34}{64}\in R$ by considering
$(\frac{5k+17}{16}, \frac{5k+17}{16}, \frac{k+5}{4},
\frac{4k^2+30k+34}{64})$. By considering
$(\frac{k+9}{4},\frac{k+9}{4},\frac{k+5}{4}, \frac{k^2+7k+18}{16})$,
we know that $\frac{k^2+7k+18}{16}\in R$. Since $1, \frac{3k+23}{16}\in
R$, it follows that $(1,\frac{k^2+7k+18}{16}, \frac{3k+23}{16},
\frac{4k^2+30k+34}{64})$ is a red solution, a contradiction.

For the lower bound, $\frac{4k^2+30k+34}{64}-1=\frac{4k^2+30k-30}{64}$.
Color $[1,\frac{k+1}{4}]\cup[\frac{k^2+7k+18}{16},\frac{4k^2+30k-30}{64}]$ red,
$[\frac{k+5}{4},\frac{k^2+7k+2}{16}]$ blue. For any
solution with $x,y,z$ all blue, we have $4w\geq \frac{k+5}{4}(k+2)$,
and so $w$ must be red. Therefore, every monochromatic solution is red.
For every such solution, $w\geq \frac{k^2+7k+18}{16}$. Note that
$z\leq \frac{k+1}{4}$. If $x, y \leq \frac{k+1}{4}$, then
$\frac{k^2+7k+18}{16} \leq w \leq
\lfloor\frac{k^2+3k+2}{16}\rfloor$, a contradiction. Thus, at least
one of $x$ or $y$ must be at least $\frac{k^2+7k+18}{16}$. If $x
\geq \frac{k^2+7k+18}{16}$, then it follows from $y+kz=4w-x$ that
$$
\frac{4(k^2+7k+18)}{16}-\frac{4k^2+30k-30}{64} \leq y+kz \leq
\frac{4(4k^2+30k-30)}{64}-\frac{k^2+7k+18}{16},
$$
i.e.,
$$
\frac{12k^2+82k+318}{64} \leq y+kz \leq \frac{12k^2+92k-192}{64}.
$$
If $z \geq \frac{3k+23}{16}$, then $y \leq
\frac{12k^2+92k-192}{64}-\frac{3k^2+23k}{16}=-3<0$, which is
impossible. If $z \leq \frac{3k+7}{16}$, then $y \geq
\frac{12k^2+82k+318}{64}-\frac{3k^2+7k}{16}=\frac{54k+318}{64}>\frac{k+1}{4}$.
Since $y \in R$, we have $y\geq \frac{k^2+7k+18}{16}$. We can assume
that
$x=\frac{4k^2+28k+a}{64},y=\frac{4k^2+28k+b}{64},w=\frac{4k^2+28k+c}{64}$,
with $a,b,c \in[72,2k-30]$. Thus,
$$
z=\frac{4w-x-y}{k}=\frac{8k^2+56k+4c-a-b}{64k}=\frac{8k+56+(4c-a-b)/k}{16}.
$$
Since $\frac{348-4k}{k} \leq \frac{4c-a-b}{k} \leq
\frac{8k-264}{k}$, it follows that $\frac{4c-a-b}{k}\in[-4,7]$, and
hence $z\in [\frac{8k+52}{64}, \frac{8k+63}{64}]$ is not an integer,
and so there are no monochromatic solutions.
\end{proof}

\begin{proposition}\label{pro3-6}
If $k\equiv 4,36 \pmod {64}$, then
$\operatorname{RR}(x+y+kz=4w)=\frac{4k^2+26k+24}{64}$.
\end{proposition}

\begin{proposition}\label{pro3-7}
If $k\equiv 20,52 \pmod {64}$, then
$\operatorname{RR}(x+y+kz=4w)=\frac{4k^2+26k+56}{64}$.
\end{proposition}

\begin{proposition}\label{pro3-8}
If $k\equiv 5,37 \pmod {64}$, then
$\operatorname{RR}(x+y+kz=4w)=\frac{4k^2+22k+46}{64}$.
\end{proposition}

\begin{proposition}\label{pro3-9}
If $k\equiv 21,53 \pmod {64}$, then
$\operatorname{RR}(x+y+kz=4w)=\frac{4k^2+22k+78}{64}$.
\end{proposition}

\begin{proposition}\label{pro3-10}
If $k\equiv 6,38 \pmod {64}$, then
$\operatorname{RR}(x+y+kz=4w)=\frac{4k^2+18k+68}{64}$.
\end{proposition}

\begin{proposition}\label{pro3-11}
If $k\equiv 22,54 \pmod {64}$, then
$\operatorname{RR}(x+y+kz=4w)=\frac{4k^2+18k+36}{64}$.
\end{proposition}

\begin{proposition}\label{pro3-12}
If $k\equiv 7 \pmod {16}$, then
$\operatorname{RR}(x+y+kz=4w)=\frac{k^2+8k+7}{16}$.
\end{proposition}

\begin{proof}
To show the upper bound, suppose that there is no monochromatic
solution for any $2$-coloring of $[1,\frac{k^2+8k+7}{16}]$. From
Lemma \ref{lem1}, we have $1, 2,...,\frac{k-7}{4}-1\in R$, then $\frac{2k+18}{16}\in R$. The solutions $(2, 3, 1, \frac{k+5}{4})$ and $(3, 4, 1, \frac{k+7}{4})$
show that $\frac{k+5}{4}, \frac{k+7}{4}\in B$. We have
$\frac{k^2+7k+14}{16}\in R$ by considering $(\frac{k+7}{4},
\frac{k+7}{4}, \frac{k+5}{4}$, $\frac{k^2+7k+14}{16})$. Using $1,
2\in R$ in $(1,1,2, \frac{2k+2}{4})$, we know that $\frac{2k+2}{4}\in B$,
and hence $\frac{k^2+8k+7}{16}\in R$ by considering
$(\frac{k+5}{4},\frac{2k+2}{4},\frac{k+5}{4},\frac{k^2+8k+7}{16})$,
then $(\frac{k^2+7k+14}{16},\frac{k^2+7k+14}{16},
\frac{2k+18}{16}, \frac{k^2+8k+7}{16})$ is a red solution, a
contradiction.

For the lower bound, $\frac{k^2+8k+7}{16}-1=\frac{k^2+8k-9}{16}$. Color $[1,\frac{k+1}{4}]\cup [\frac{k^2+7k+14}{16},\frac{k^2+8k-9}{16}]$ red,
$[\frac{k+5}{4},\frac{k^2+7k+14}{16}]$ blue. For any solution
with $x,y,z$ all blue, we know that $4w\geq \frac{k+5}{4}(k+2)$, so $w$
must be red. Clearly, every monochromatic solution is red. For every
such solution, $w\geq \frac{k^2+7k+14}{16}$and then $z\leq
\frac{k+1}{4}$, it follows that if $x, y \leq \frac{k+1}{4}$, then
$\frac{k^2+7k+14}{16} \leq w \leq
\lfloor\frac{k^2+3k+2}{16}\rfloor$, a contradiction. Therefore, at least
one of $x$ or $y$ must be at least $\frac{k^2+7k+14}{16}$. If $x
\geq \frac{k^2+7k+14}{16}$, then it follows from $y+kz=4w-x$ that
$$
\frac{4(k^2+7k+14)}{16}-\frac{k^2+8k-9}{16} \leq y+kz \leq
\frac{4(k^2+8k-9)}{16}-\frac{k^2+7k+14}{16},
$$
i.e.,
$$
\frac{3k^2+20k+73}{16} \leq y+kz \leq \frac{3k^2+25k-50}{16}.
$$
If $z \geq \frac{3k+27}{16}$, then $y \leq
\frac{3k^2+25k-50}{16}-\frac{3k^2+27k}{16}=\frac{-2k-50}{16}<0$,
which is impossible. If $z \leq \frac{3k+11}{16}$, then $y \geq
\frac{3k^2+20k+73}{16}-\frac{3k^2+11k}{16}=\frac{9k+73}{16}>\frac{k+1}{4}$.
Since $y \in R$, it follows that $y\geq \frac{k^2+7k+14}{16}$. We can assume
that
$x=\frac{k^2+7k+a}{16},y=\frac{k^2+7k+b}{16},w=\frac{k^2+7k+c}{16}$,
with $a,b,c \in[14,k-9]$. Thus,
$$
z=\frac{4w-x-y}{k}=\frac{2k^2+14k+4c-a-b}{16k}=\frac{2k+14+(4c-a-b)/k}{16},
$$
Since $\frac{74-2k}{k} \leq \frac{4c-a-b}{k} \leq \frac{4k-64}{k}$,
it follows that $\frac{4c-a-b}{k}\in[-2,3]$, and hence $z\in
[\frac{2k+12}{16}, \frac{2k+17}{16}]$ is not an integer, and so
there are no monochromatic solutions.
\end{proof}

\begin{proposition}\label{pro3-13}
If $k\equiv 8 \pmod {16}$, then
$\operatorname{RR}(x+y+kz=4w)=\frac{k^2+7k+8}{16}$.
\end{proposition}

\begin{proposition}\label{pro3-14}
If $k\equiv 9 \pmod {16}$, then
$\operatorname{RR}(x+y+kz=4w)=\frac{k^2+6k+9}{16}$.
\end{proposition}

\begin{proposition}\label{pro3-15}
If $k\equiv 10 \pmod {16}$, then
$\operatorname{RR}(x+y+kz=4w)=\frac{k^2+5k+10}{16}$.
\end{proposition}

\begin{proposition}\label{pro3-16}
If $k\equiv 11 \pmod {16}$, then
$\operatorname{RR}(x+y+kz=4w)=\frac{k^2+9k+20}{16}$.
\end{proposition}

\begin{proof}
To show the upper bound, suppose that there is no monochromatic
solution for any $2$-coloring of $[1,\frac{k^2+9k+20}{16}]$. From
Lemma \ref{lem1}, we have $1, 2,...,\frac{k-11}{4}-1\in R$, then $\frac{2k+10}{16}\in R$. The
solution $(2, 3, 1, \frac{k+5}{4})$ show that $\frac{k+5}{4}\in B$.
We know that $\frac{k^2+7k+14}{16}\in R$ by considering
$(\frac{k+5}{4}, \frac{k+5}{4}, \frac{k+5}{4},
\frac{k^2+7k+10}{16})$. Using $2,5\in R$ in $(5,5,2,
\frac{2k+10}{4})$, we have $\frac{2k+10}{4}\in B$, then
$\frac{k^2+9k+20}{16}\in R$ by considering
$(\frac{2k+10}{4},\frac{2k+10}{4},\frac{k+5}{4},\frac{k^2+9k+20}{16})$,
and so $(\frac{k^2+9k+20}{16},\frac{k^2+9k+20}{16},
\frac{2k+10}{16}, \frac{k^2+7k+10}{16})$ is a red solution, a
contradiction.

For the lower bound, $\frac{k^2+9k+20}{16}-1=\frac{k^2+9k+4}{16}$.
Color $[1,\frac{k+1}{4}]\cup [\frac{k^2+7k+10}{16},\frac{k^2+9k+4}{16}]$ red, and
$[\frac{k+5}{4},\frac{k^2+7k+10}{16}]$ blue. For any solution
with $x,y,z$ all blue, we have $4w\geq \frac{k+5}{4}(k+2)$, so $w$
must be red. Clearly, every monochromatic solution is red. For every
such solution, $w\geq \frac{k^2+7k+10}{16}$. Recall that $z\leq
\frac{k+1}{4}$. If $x, y \leq \frac{k+1}{4}$, then
$\frac{k^2+7k+10}{16} \leq w \leq
\lfloor\frac{k^2+3k+2}{16}\rfloor$, a contradiction. Thus, at least
one of $x$ or $y$ must be at least $\frac{k^2+7k+10}{16}$. If $x
\geq \frac{k^2+7k+10}{16}$, then it follows from $y+kz=4w-x$ that
$$
\frac{4(k^2+7k+10)}{16}-\frac{k^2+9k+4}{16} \leq y+kz \leq
\frac{4(k^2+9k+4)}{16}-\frac{k^2+7k+10}{16},
$$
i.e.,
$$
\frac{3k^2+19k+36}{16} \leq y+kz \leq \frac{3k^2+29k+6}{16}.
$$
If $z \geq \frac{3k+31}{16}$, then $y \leq
\frac{3k^2+29k+6}{16}-\frac{3k^2+31k}{16}=\frac{-2k+6}{16}<0$, which
is impossible. If $z \leq \frac{3k+15}{16}$, then $y \geq
\frac{3k^2+29k+6}{16}-\frac{3k^2+15k}{16}=\frac{14k+6}{16}>\frac{k+1}{4}$.
Since $y \in R$, we have $y\geq \frac{k^2+7k+10}{16}$. We can assume
that
$x=\frac{k^2+7k+a}{16},y=\frac{k^2+7k+b}{16},w=\frac{k^2+7k+c}{16}$,
with $a,b,c \in[10,2k+4]$. Thus, we have
$$
z=\frac{4w-x-y}{k}=\frac{2k^2+14k+4c-a-b}{16k}=\frac{2k+14+(4c-a-b)/k}{16},
$$
Since $\frac{32-4k}{k} \leq \frac{4c-a-b}{k} \leq \frac{8k-4}{k}$,
it follows that $\frac{4c-a-b}{k}\in[-3,7]$, and hence $z\in
[\frac{2k+11}{16}, \frac{2k+21}{16}]$ is not an integer, and so
there are no monochromatic solutions.
\end{proof}

\begin{proposition}\label{pro3-17}
If $k\equiv 12 \pmod {16}$, then
$\operatorname{RR}(x+y+kz=4w)=\frac{k^2+8k+16}{16}$.
\end{proposition}

\begin{proposition}\label{pro3-18}
If $k\equiv 13 \pmod {16}$, then
$\operatorname{RR}(x+y+kz=4w)=\frac{k^2+7k+12}{16}$.
\end{proposition}

\begin{proposition}\label{pro3-19}
If $k\equiv 14 \pmod {16}$, then
$\operatorname{RR}(x+y+kz=4w)=\frac{k^2+6k+8}{16}$.
\end{proposition}

\begin{proposition}\label{pro3-20}
If $k\equiv 15 \pmod {16}$, then
$\operatorname{RR}(x+y+kz=4w)=\frac{k^2+7k+22}{16}$.
\end{proposition}
\begin{proof}
To show the upper bound, we suppose that there is no monochromatic
solution for any $2$-coloring of $[1,\frac{k^2+7k+22}{16}]$. From
Lemma \ref{lem1}, we have $1, 2,...,\frac{k-3}{4}\in R$. The
solutions $(3, 2, 1, \frac{k+5}{4})$ and $(6, 5, 1, \frac{k+11}{4})$
show that $\frac{k+5}{4}, \frac{k+11}{4}\in B$. Then
$\frac{k^2+7k+22}{16}\in R$ by considering
$(\frac{k+11}{4},\frac{k+11}{4},\frac{k+5}{4},\frac{k^2+7k+22}{16})$,
and hence $(\frac{k^2+7k+22}{16},\frac{2k+66}{16}, \frac{3k+19}{16},
\frac{k^2+7k+22}{16})$ is a red solution.

For the lower bound, $\frac{k^2+7k+22}{16}-1=\frac{k^2+7k+6}{16}$. Color all
elements in $[1, \frac{k+1}{4}]$ red and the remaining elements
blue. If $x, y, z$ are all red, then $x+y+kz\geq k+2$, and so $w\geq
\lceil\frac{k+2}{4}\rceil=\frac{k+5}{4}$, and hence there is no red
solution. If $x, y, z$ are all blue, then $x+y+kz\geq
\frac{k^2+7k+10}{4}$, and so $w\geq
\lceil\frac{k^2+7k+10}{16}\rceil=\frac{k^2+7k+22}{16}> M-1$ and
hence there is no blue solution.
\end{proof}

\noindent\textbf{Proof of Theorem \ref{th-Exact-1}}: If $6\leq k<32$
and $k\neq 16$, the values of $\operatorname{RR}(x+y+kz=4w)$, given
in \cite{RoMy08}, are in the Table $2$, and they all satisfy the
formulas in Theorem \ref{th-Exact-1}. For $k\geq 32$, the result
follows from Propositions \ref{pro3-1} to \ref{pro3-20}.
\begin{table}[H]
\caption{Some values of $\operatorname{RR}(x+y+kz=4 w)$.}
\begin{center}
\begin{tabularx}{15cm}{p{0.8cm}<{}
  p{0.05cm}<{} p{0.05cm}<{} p{0.05cm}<{} p{0.05cm}<{} p{0.15cm}<{} p{0.15cm}<{}  p{0.12cm}<{} p{0.12cm}<{} p{0.12cm}<{} p{0.12cm}<{} p{0.12cm}<{} p{0.12cm}<{} p{0.12cm}<{} p{0.12cm}<{} p{0.12cm}<{} p{0.12cm}<{} p{0.12cm}<{} p{0.12cm}<{}  p{0.12cm}<{} p{0.12cm}<{} p{0.12cm}<{} p{0.12cm}<{} p{0.12cm}<{} p{0.12cm}<{} p{0.12cm}<{}  }
\hline
$k$ &6&7&8&9&10&11&12&13&14&15
&17&18&19&20&21&22&23&24&25&26&27&28&29&30&31  \\
\hline
value   & $5$ & $7$ & $8$ & $9$ & $10$ & $15$ & $16$ & $17$ & $18$  & $24$ & $24$  & $25$ & $32$ & $34$ & $36$ & $37$ & $45$ & $47$ & $49$ & $51$ & $62$ & $64$ & $66$ & $68$ & $75$ \\
\hline
\end{tabularx}
\end{center}
\end{table}

\section{Results for five variable Rado numbers}

In this section, we let $\mathcal{E}(5,k,k+j)$ denote the equation
$x+y+z+kv=(k+j)w$. We now determine the exact values for
$\operatorname{RR}(\mathcal{E}(5,k,k+j))$ for $1\leq j\leq 5$. We
call a coloring of $[1,n]$ \emph{valid} if it contains no
monochromatic solution to $\mathcal{E}(5,k,k+j)$.

It is clear that $\operatorname{RR}(\mathcal{E}(5,k,3))=1$ for all
$k\in \mathbb{Z}^+$.

\begin{theorem}
For $k \in \mathbb{Z}^+$,
$\operatorname{RR}(\mathcal{E}(5,k,2))=\operatorname{RR}(\mathcal{E}(5,k,4))=4$.
\end{theorem}
\begin{proof}
Assume, for a contradiction, that there exists a $2$-coloring of
$[1,4]$ with no monochromatic solution to $x+y+z+kv=(k+2)w$. We
assume that $1$ is red. Considering the solutions $(1,1,2,2,2),
(2,2,2,3,3), (1,1,4,3,3)$, in order, we find that $2$ is blue, $3$
is red, $4$ is blue. But $(2,2,4,4,4)$ is a blue solution, a
contradiction. Thus $\operatorname{RR}(\mathcal{E}(5,k,2))\leq 4$
for all $k \in \mathbb{Z}^+$.

It is easy to find that the valid coloring of $[1,3]$ is $rbr$. If $x, y, z, v$ are all red, then $x+y+z+kv\in \{k+3, k+5, k+7, k+9, 3k+3, 3k+5, 3k+7,
3k+9\}$. If $x, y, z, v$ are all blue, then $x+y+z+kv=2k+6$. If $w$
is red, then $(k+2)w\in \{k+2, 3k+6\}$. If $w$ is blue, then
$(k+2)w=2k+4$. We denote these results by
$$
R_{x,y,z,v}=\{k+3, k+5, k+7, k+9, 3k+3, 3k+5, 3k+7, 3k+9\},
$$
$$
R_{w}=\{k+2, 3k+6\}, B_{w}=\{k+2, 3k+6\}, B_{x,y,z,v}=\{2k+6\}.
$$
Then $R_{x,y,z,v}\cap R_w=\emptyset$ and $B_{x,y,z,v}\cap
B_w=\emptyset$, and hence $\operatorname{RR}(\mathcal{E}(5,k,2))=4$
for all $k \in \mathbb{Z}^+$.

Similarly, $\operatorname{RR}(\mathcal{E}(5,k,4))=4$ for all $k \in
\mathbb{Z}^+$.
\end{proof}

\begin{theorem}
For $k \in \mathbb{Z}^+$,
$$
\operatorname{RR}(\mathcal{E}(5,k,5))=
\begin{cases} 3& \text{for $k=1, 2, 3, 4$},\\
5 & \text{for $5\leq k\leq 11$,and $k=13, 14, 17$},\\
6 & \text{for $k=12, 15, 16, 18, 19, 22$},\\
8 & \text{for $k=21, 23, 24$,and $26\leq k\leq 29$},\\
9 & \text{for $k=20, 25$ and $k\geq 30$}.\\
\end{cases}
$$
\end{theorem}
\begin{proof}
Assume, for a contradiction, that there exists a $2$-coloring of
$[1, 9]$ with no monochromatic solution to $x+y+z+kv=(k+5)w$.
Without loss of generality, we assume that $1\in R$. Considering
the solutions $(2,2,1,1,1)$, $(3,1,1,1,1)$, $(4,4,2,2,2)$, $(6,2,2,2,2)$ and
$(9,3,3,3,3)$, one can see $2,3\in B$ and $4,6,9\in R$. Since the solution $(6,9,5,4,4)$ is not monochromatic solution, it implies that $5\in B$, which contradicts to the fact that
$(5,5,5,3,3)$ is a blue solution. Hence, we have
$\operatorname{RR}(\mathcal{E}(k,5))\leq 9$, for all $k \in
\mathbb{Z}^+$.

It is easy to check that the only valid $2$-colorings (using $r$ for
red, $b$ for blue, and assuming that $1$ is red) of $[1,n]$ for
$n=3, 5, 6, 8, 9$ are in Table $3$ by using Matlab.

\begin{table}[htbp]
\caption{The valid $2$-colorings of $[1,n]$.}
\begin{center}
\begin{tabularx}{10cm}{p{1cm}<{} p{3cm}<{}}
\hline
$n$ & Valid colorings   \\
\hline
$3$  & $rbb $  \\
$5$  & $rbbrr$ \\
$6$  & $rbbrrr$ \\
$8$  & $rbbrrrbb$ \\
\hline
\end{tabularx}
\end{center}
\end{table}

For $n=3$, we consider the valid coloring $rbb$. If $x, y, z, v$ are
all red, then $x+y+z+kv=k+3$. If $x, y, z, v$ are all blue, then
$x+y+z+kv\in \{2k+6, 2k+7, 2k+8, 2k+9, 3k+6, 3k+7, 3k+8, 3k+9\}$. If
$w$ is red, then $(k+5)w=k+5$. If $w$ is blue, then $(k+5)w\in
\{2k+10, 3k+15\}$. We denote these results by
$$
R_{x,y,z,v}=\{k+3\}, R_{w}=\{k+5\}, B_{w}=\{2k+10, 3k+15\},
$$
and
$$
B_{x,y,z,v}=\{2k+6, 2k+7, 2k+8, 2k+9, 3k+6, 3k+7, 3k+8, 3k+9\}.
$$

We see that $R_{x,y,z,v}\cap R_w= \emptyset$ and $B_{x,y,z,v}\cap
B_w\neq \emptyset$ only when $k=4 \ (2k+10=3k+6)$, or $k=1 \
(2k+10=3k+7)$, or $k=2 \ (2k+10=3k+8)$, or $k=1 \ (2k+10=3k+9)$.
Thus, we have $\operatorname{RR}(\mathcal{E}(k,5))=3$ for
$k=1,2,3,4$.

For $n=5$, by considering the valid coloring $rbbrr$, we see that
$$
R_{x,y,z,v}=\{ik+j:i=1,4,5; j=3,6,7,9,10,...,15\};
R_{w}=\{i(k+5):i=1,4,5\}
$$
$$
B_{x,y,z,v}=\{ik+j:i=2,3; j=6,7,8,9\}; B_{w}=\{i(k+5):i=2,3\}.
$$
It is clear that $R_{x,y,z,v}\cap R_{w}\neq\emptyset$ only when
$k=20-j \ (4k+20=5k+j)$, $j=3,6,7,9,10,\ldots,15$.
Then $B_{x,y,z,v}\cap B_{w}\neq \emptyset$ only when $k=10-j \
(2k+10=3k+j), j=6,7,8,9$, and hence $k=1,2,\ldots,11,13,14,17$. Therefore, we have
$\operatorname{RR}(\mathcal{E}(k,5))=5$, for $k=5,6,\ldots,11,13,14,17$.

For $n=6$, by considering the valid coloring $rbbrrr$, we see that
$$
R_{x,y,z,v}=\{ik+j:i=1,4,5,6; j=3,6,7,\ldots,18\};
R_{w}=\{i(k+5):i=1,4,5,6\}
$$
$$
B_{x,y,z,v}=\{ik+j:i=2,3; j=6,7,8,9\}; B_{w}=\{i(k+5):i=2,3\}.
$$
It is clear that $R_{x,y,z,v}\cap R_{w}\neq\emptyset$ only when
$k=20-j \ (4k+20=5k+j)$, $j=3,6,7,\ldots,18$ or $k=\frac{20-j}{2} \ (4k+20=6k+j)$, $j=6,8,10,12,14,16,18$ or $k=25-j \ (5k+25=6k+j)$, $j=3,6,7,\ldots,18$.
Then $B_{x,y,z,v}\cap B_{w}\neq \emptyset$ only when $k=10-j \
(2k+10=3k+j), j=6,7,8,9$, and hence $k=1,2,\ldots,19,22$. Thus, we have
$\operatorname{RR}(\mathcal{E}(k,5))=6$, for $k=12,15,16,18,19,22$.

For $n=8$, by considering the valid coloring $rbbrrrbb$, we see that
$$
R_{x,y,z,v}=\{ik+j:i=1,4,5,6; j=3,6,7,\ldots,18\};
R_{w}=\{i(k+5):i=1,4,5,6\}
$$
$$
B_{x,y,z,v}=\{ik+j:i=2,3,7,8; j\in [6,24]\backslash\{10,15,20\}\}; B_{w}=\{i(k+5):i=2,3,7,8\}.
$$
It is clear that $R_{x,y,z,v}\cap R_{w}\neq\emptyset$ only when
$k=20-j \ (4k+20=5k+j)$, $j=3,6,7,\ldots,18$ or $k=\frac{20-j}{2} \ (4k+20=6k+j)$, $j=6,8,10,12,14,16,18$ or $k=25-j \ (5k+25=6k+j)$, $j=3,6,7,\ldots,18$.
Then $B_{x,y,z,v}\cap B_{w}\neq \emptyset$ only when
$k=10-j \ (2k+10=3k+j), j=6,7,8,9$, or $k=\frac{15-j}{4} \ (3k+15=7k+j), j=7$,
or $k=35-j \ (7k+35=8k+j), j\in [6,24]\backslash\{10,15,20\}$,
or $k=j-15 \ (3k+15=2k+j), j\in [16,24]\backslash\{20\}$
 and hence $k\in [1,29]\backslash \{20,25\}$. Thus, we have
$\operatorname{RR}(\mathcal{E}(k,5))=8$, for $k=21,23,24,26,\ldots,29$.

\end{proof}

We now give the exact value for
$\operatorname{RR}(\mathcal{E}(5,k,1))$, and its proof is in
Appendix II.
\begin{theorem}\label{th-Last}
For $k \in \mathbb{Z}^+$,
$$
\operatorname{RR}(\mathcal{E}(5,k,1))=
\begin{cases} 3& \text{for $k=3, 4$},\\
4 & \text{for $k=5, 6$},\\
5 & \text{for $k=1, 2, 7, 8, 9$},\\
6 & \text{for $k=10, 11, 12$},\\
7 & \text{for $k=13, 14$},\\
8 & \text{for $k=15,16$},\\
9 & \text{for $k=17, 18$},\\
10 & \text{for $k=19, 20$},\\
11 & \text{for $k\geq 21$}.\\
\end{cases}
$$
\end{theorem}

\bibliographystyle{line}

\section{(For referee) Appendix I: Proof of Theorem \ref{th-Exact-1}}

The proof of Proposition \ref{pro3-2} is:
\begin{proof}
To show the upper bound, we suppose that there is no monochromatic
solution for any $2$-coloring of $[1,\frac{k^2+5k+10}{16}]$. From
Lemma \ref{lem1}, we have $1, 2,...,\frac{k-5}{4}\in R$ and that $k,
2k,..., \frac{k^2-k}{16}\in B$. So $(1, 2, 1, \frac{k+3}{4})$ and
$(3, 4, 1, \frac{k+7}{4})$ show that $\frac{k+3}{4}\in B$ and
$\frac{k+7}{4}\in B$. In turn we have $\frac{k^2+5k+10}{16}\in R$ by
considering $(\frac{k+3}{4}, \frac{k+7}{4}, \frac{k+3}{4},
\frac{k^2+5k+10}{16})$. Using $\frac{2k+30}{16}\in R$ in
$(\frac{k^2+5k+10}{16},$ $\frac{k^2+3k+20}{16}, \frac{2k+14}{16},
\frac{k^2+5k+10}{16})$ we have $\frac{k^2+5k+10}{16}\in B$.
Furthermore, we have $\frac{2k+10}{4}\in B$ by considering $(5, 5,
2, \frac{2k+10}{4})$, and so $(\frac{2k+10}{4}, \frac{2k+10}{4},
\frac{k-1}{4}, \frac{k^2+3k+20}{16})$ shows that $\frac{k-1}{4}\in
R$, and hence $\frac{k^2+k-2}{16}\in B$ by considering
$(\frac{k-1}{4}, \frac{k-1}{4}, \frac{k-1}{4},$
$\frac{k^2+k-2}{16})$. Using $\frac{k^2-k}{16}\in B$ in $(k, 2k,
\frac{k-13}{4}, \frac{k^2-k}{16})$, we have $\frac{k-25}{4}\in R$, then $\frac{k^2-12k-21}{16}\in B$ by considering
$(\frac{k-25}{4}, 1, \frac{k-13}{4}, \frac{k^2-12k-21}{16})$.  Therefore, $(\frac{k^2-12k-21}{16},
\frac{3k+13}{16},$ $ \frac{3k+13}{16}, \frac{k^2+k-2}{16})$ is a
blue solution, a contradiction.

To prove the lower bounds, $\frac{k^2+5k+10}{16}-1=\frac{k^2+5k-6}{16}$. We color the integers in $[1,\frac{k-1}{4}]$ red and the integers in
$[\frac{k+3}{4},\frac{k^2+5k-6}{16}]$ blue. Let $x,y,z,w$
be a monochromatic solution of $x+y+kz=4w$. If $x, y, z$ are all
red, then $x+y+kz\geq k+2$, and hence $w\geq
\lceil\frac{k+2}{4}\rceil=\frac{k+3}{4}$, and so there is no red
solution. If $x, y, z$ are all blue, then $x+y+kz\geq
\frac{k^2+5k+6}{4}$, and so $w\geq
\lceil\frac{k^2+5k+6}{16}\rceil=\frac{k^2+5k+10}{16}> \frac{k^2+5k-6}{16}$, and
hence there is no blue solution.
\end{proof}

The proof of Proposition \ref{pro3-3} is:
\begin{proof}
To show the upper bound, we suppose that there is no monochromatic
solution for any $2$-coloring of $[1,\frac{k^2+4k+4}{16}]$. From
Lemma \ref{lem1}, we have $1, 2,...,\frac{k-2}{4}-1\in R$ and that
$k, 2k,..., \frac{k^2-2k}{16}\in B$. So $(1, 1, 1, \frac{k+2}{4})$
show that $\frac{k+2}{4}\in B$. In turn we have
$\frac{k^2+4k+4}{16}\in R$ by considering $(\frac{k+2}{4},
\frac{k+2}{4}, \frac{k+2}{4}, \frac{k^2+4k+4}{16})$ For our
contradiction, we see now that
$(\frac{k^2+4k+4}{16},\frac{2k+12}{16}, \frac{3k+10}{16},
\frac{k^2+4k+4}{16})$ is a red solution.

For the lower bound, $\frac{k^2+4k+4}{16}-1=\frac{k^2+4k-12}{16}$. We color the integers in $[1,\frac{k-2}{4}]$ red and the integers in
$[\frac{k+2}{4},\frac{k^2+4k-12}{16}]$ blue.
If $x, y, z$ are all red, then $x+y+kz\geq k+2$, and
so $w\geq \lceil\frac{k+2}{4}\rceil=\frac{k+2}{4}$, and hence there
is no red solution. If $x, y, z$ are all blue, then $x+y+kz\geq
\frac{k^2+4k+4}{4}$, and so $w\geq
\lceil\frac{k^2+4k+4}{16}\rceil=\frac{k^2+4k+4}{16}> \frac{k^2+4k-12}{16}$ and hence
there is no blue solution.
\end{proof}

The proof of Proposition \ref{pro3-6} is:
\begin{proof}
If $k\geq 48$, to show the upper bound, suppose that there is no monochromatic
solution for any $2$-coloring of $[1,\frac{4k^2+30k+34}{64}]$. From
Lemma \ref{lem1}, we have $1, 2,...,\frac{k-4}{4}-1\in R$. The
solutions $(5, \frac{k-8}{4}, 1, \frac{5k+12}{16})$ and
$(2,2,1,\frac{k+4}{4})$ show that $\frac{5k+12}{16},\frac{k+4}{4}\in
B$. In turn we have $\frac{4k^2+26k+24}{64}\in R$ by considering
$(\frac{5k+12}{16}, \frac{5k+12}{16}, \frac{k+4}{4},
\frac{4k^2+26k+24}{64})$. By considering
$(\frac{k+4}{4},\frac{k+4}{4},\frac{k+4}{4}, \frac{k^2+6k+8}{16})$
we have $\frac{k^2+6k+8}{16}\in R$. Since $1, \frac{3k+20}{16}\in
R$, we see that $(1,\frac{k^2+6k+8}{16}, \frac{3k+20}{16},
\frac{4k^2+26k+24}{64})$ is a red solution, a contradiction.

For the lower bound, $\frac{4k^2+26k+24}{64}-1=\frac{4k^2+26k-40}{64}$. Color $[1,\frac{k}{4}]\cup[\frac{k^2+6k+8}{16},\frac{4k^2+26k-40}{64}]$ red and
$[\frac{k+4}{4},\frac{k^2+6k-8}{16}]$ blue. For any solution
with $x,y,z$ all blue, we have $4w\geq \frac{k+4}{4}(k+2)$, and so
$w$ must be red. Thus, every monochromatic solution is red. For
every such solution, $w\geq \frac{k^2+6k+8}{16}$. We have $z\leq
\frac{k}{4}$, so if $x, y \leq \frac{k}{4}$ then
$\frac{k^2+6k+8}{16} \leq w \leq \lfloor\frac{k^2+2k}{16}\rfloor$, a
contradiction. Therefore, at least one of $x$ or $y$ must be at
least $\frac{k^2+6k+8}{16}$. If $x \geq \frac{k^2+6k+8}{16}$, then
it follows from $y+kz=4w-x$ that
$$\frac{4(k^2+6k+8)}{16}-\frac{4k^2+26k-40}{64} \leq y+kz \leq
\frac{4(4k^2+26k-40)}{64}-\frac{k^2+6k+8}{16},$$
i.e.,
$\frac{12k^2+70k+168}{64} \leq y+kz \leq \frac{12k^2+80k-192}{64}.$
If $z \geq \frac{3k+20}{16}$, then $y \leq
\frac{12k^2+80k-192}{64}-\frac{3k^2+20k}{16}=-3<0$, which is
impossible. If $z \leq \frac{3k+4}{16}$, then $y \geq
\frac{12k^2+70k+168}{64}-\frac{3k^2+4k}{16}=\frac{54k+168}{64}>\frac{k}{4}$.
Since $y \in R$, we have $y\geq \frac{k^2+6k+8}{16}$. We can assume
that
$x=\frac{4k^2+24k+a}{64},y=\frac{4k^2+24k+b}{64},w=\frac{4k^2+24k+c}{64}$,
with $a,b,c \in[32,2k-40]$. Thus,
$z=\frac{4w-x-y}{k}=\frac{8k^2+48k+4c-a-b}{64k}=\frac{8k+48+(4c-a-b)/k}{16}.$
Since $\frac{168-4k}{k} \leq \frac{4c-a-b}{k} \leq
\frac{8k-224}{k}$, it follows that $\frac{4c-a-b}{k}\in[-4,7]$, and
hence $z\in [\frac{8k+44}{64}, \frac{8k+55}{64}]$ is not an integer,
and so there are no monochromatic solutions.

If $16<k<48$, we know that $k=36$. To show the upper bound, we
suppose that there is no monochromatic solution for any $2$-coloring of
$[1,96]$. From Lemma \ref{lem1}, we have $1,2,...,7\in R$. The solutions
$(2,2,1,10)$, $(7,5,1,12)$, $(26,6,2,26)$, $(85,3,7,85)$ show that $10,12,26,85\in B$. Then we have $8,95,96\in R$ by
considering $(26,26,8,85)$, $(10,10,10,95)$ and $(12,12,10,96)$.
Therefore, $(3,95,8,96)$ is a red solution, a contradiction.

For the lower bound, $96-1=95$,
we color $[1,9]\cup\{95\}$ red,
$[10,94]$ blue. It is easy to proof that there are no monochromatic solutions.

\end{proof}

The proof of Proposition \ref{pro3-7} is:
\begin{proof}
To show the upper bound, we suppose that there is no monochromatic
solution for any $2$-coloring of $[1,\frac{4k^2+26k+56}{64}]$. From
Lemma \ref{lem1}, we have $1, 2,...,\frac{k-4}{4}-1\in R$. The
solutions $(8, \frac{k-8}{4}, 1, \frac{5k+28}{16})$ and
$(2,2,1,\frac{k+4}{4})$ show that $\frac{5k+28}{16},\frac{k+4}{4}\in
B$. In turn we have $\frac{4k^2+26k+56}{64}\in R$ by considering
$(\frac{5k+28}{16}, \frac{5k+28}{16}, \frac{k+4}{4},
\frac{4k^2+26k+56}{64})$. By considering
$(\frac{k+4}{4},\frac{k+4}{4},\frac{k+4}{4}, \frac{k^2+6k+8}{16})$
we have $\frac{k^2+6k+8}{16}\in R$. Since $1, \frac{3k+20}{16}\in
R$, one can see that $(3,\frac{k^2+6k+8}{16}, \frac{3k+20}{16},
\frac{4k^2+26k+56}{64})$ is a red solution, a contradiction.

For the lower bound, $\frac{4k^2+26k+56}{64}-1=\frac{4k^2+26k-8}{64}$. Color $[1,\frac{k}{4}]\cup [\frac{k^2+6k+8}{16},\frac{4k^2+26k-8}{64}]$ red and $[\frac{k+4}{4},\frac{k^2+6k-8}{16}]$
blue. For any
solution with $x,y,z$ all blue, we have $4w\geq \frac{k+4}{4}(k+2)$,
so $w$ must be red. Thus, every monochromatic solution is red. For
every such solution, $w\geq \frac{k^2+6k+8}{16}$. Recall that $z\leq
\frac{k}{4}$. If $x, y \leq \frac{k}{4}$, then $\frac{k^2+6k+8}{16}
\leq w \leq \lfloor\frac{k^2+2k}{16}\rfloor$, a contradiction. Thus,
at least one of $x$ or $y$ must be at least $\frac{k^2+6k+8}{16}$.
If $x \geq \frac{k^2+6k+8}{16}$, since $y+kz=4w-x$, then
$$
\frac{4(k^2+6k+8)}{16}-\frac{4k^2+26k-8}{64} \leq y+kz \leq
\frac{4(4k^2+26k-8)}{64}-\frac{k^2+6k+8}{16},$$
i.e.,$\frac{12k^2+70k+136}{64} \leq y+kz \leq \frac{12k^2+80k-64}{64}.$
If $z \geq \frac{3k+20}{16}$, then $y \leq
\frac{12k^2+80k-64}{64}-\frac{3k^2+20k}{16}=-1<0$, which is
impossible. If $z \leq \frac{3k+4}{16}$, then $y \geq
\frac{12k^2+70k+136}{64}-\frac{3k^2+4k}{16}=\frac{54k+136}{64}>\frac{k}{4}$.
Since $y \in R$, we have $y\geq \frac{k^2+6k+8}{16}$. Assume
that
$x=\frac{4k^2+24k+a}{64},y=\frac{4k^2+24k+b}{64},w=\frac{4k^2+24k+c}{64}$,
with $a,b,c \in[32,2k-8]$. Thus,$z=\frac{4w-x-y}{k}=\frac{8k^2+48k+4c-a-b}{64k}=\frac{8k+48+(4c-a-b)/k}{16}.$
Since $\frac{144-4k}{k} \leq \frac{4c-a-b}{k} \leq \frac{8k-96}{k}$,
follows that $\frac{4c-a-b}{k}\in[-4,7]$, and hence $z\in
[\frac{8k+44}{64}, \frac{8k+55}{64}]$ is not an integer, and so
there are no monochromatic solutions.
\end{proof}

The proof of Proposition \ref{pro3-8} is:
\begin{proof}
To show the upper bound, we suppose that there is no monochromatic
solution for any $2$-coloring of $[1,\frac{4k^2+22k+46}{64}]$. From
Lemma \ref{lem1}, we have $1, 2,...,\frac{k-1}{4}-1\in R$. The
solutions $(7, \frac{k-5}{4}, 1, \frac{5k+23}{16})$,
$(2,1,1,\frac{k+3}{4})$, and $(3,4,1,\frac{k+7}{4})$ show that
$\frac{5k+23}{16},\frac{k+3}{4},\frac{k+7}{4}\in B$. In turn we have
$\frac{4k^2+26k+46}{64}\in R$ by considering $(\frac{5k+23}{16},
\frac{5k+23}{16}, \frac{k+3}{4}, \frac{4k^2+22k+46}{64})$.
Considering $(\frac{k+7}{4},\frac{k+7}{4},\frac{k+3}{4},
\frac{k^2+5k+14}{16})$ we have $\frac{k^2+5k+14}{16}\in R$. Since
$2, \frac{3k+17}{16}\in R$, we see that $(2,\frac{k^2+5k+14}{16},
\frac{3k+17}{16}, \frac{4k^2+22k+46}{64})$ is a red solution, a
contradiction.

For the lower bound, $\frac{4k^2+22k+46}{64}-1=\frac{4k^2+22k-18}{64}$. Color $[1,\frac{k-1}{4}]\cup [\frac{k^2+5k+14}{16},\frac{4k^2+22k-18}{64}]$ red and
$[\frac{k+3}{4},\frac{k^2+5k-2}{16}]$ blue. For any
solution with $x,y,z$ all blue, we have $4w\geq \frac{k+3}{4}(k+2)$,
and so $w$ must be red. Thus, every monochromatic solution is red.
For every such solution, $w\geq \frac{k^2+5k+14}{16}$. We have
$z\leq \frac{k+1}{4}$, so if $x, y \leq \frac{k+1}{4}$ then
$\frac{k^2+5k+14}{16} \leq w \leq
\lfloor\frac{k^2+3k+2}{16}\rfloor$, a contradiction. Therefore, at
least one of $x$ or $y$ must be at least $\frac{k^2+5k+14}{16}$. If
$x \geq \frac{k^2+5k+14}{16}$, then it follows from $y+kz=4w-x$ that$$
\frac{4(k^2+5k+14)}{16}-\frac{4k^2+22k-18}{64} \leq y+kz \leq
\frac{4(4k^2+22k-18)}{64}-\frac{k^2+5k+14}{16},$$
i.e.,$
\frac{12k^2+58k+242}{64} \leq y+kz \leq \frac{12k^2+68k-128}{64}.$
If $z \geq \frac{3k+17}{16}$, then $y \leq
\frac{12k^2+68k-128}{64}-\frac{3k^2+68k}{16}=-2<0$, which is
impossible. If $z \leq \frac{3k+1}{16}$, then $y \geq
\frac{12k^2+58k+242}{64}-\frac{3k^2+k}{16}=\frac{54k+242}{64}>\frac{k-1}{4}$.
Since $y \in R$, we have $y\geq \frac{k^2+5k+14}{16}$. We can assume
that
$x=\frac{4k^2+20k+a}{64},y=\frac{4k^2+20k+b}{64},w=\frac{4k^2+20k+c}{64}$,
with $a,b,c \in[14,2k-18]$. Thus, we have$
z=\frac{4w-x-y}{k}=\frac{8k^2+40k+4c-a-b}{64k}=\frac{8k+40+(4c-a-b)/k}{16}.$
Since $\frac{92-4k}{k} \leq \frac{4c-a-b}{k} \leq \frac{8k-100}{k}$,
it follows that $\frac{4c-a-b}{k}\in[-4,7]$, and hence $z\in
[\frac{8k+36}{64}, \frac{8k+47}{64}]$ is not an integer, and so
there are no monochromatic solutions.
\end{proof}

The proof of Proposition \ref{pro3-9} is:
\begin{proof}
To show the upper bound, we suppose that there is no monochromatic
solution for any $2$-coloring of $[1,\frac{4k^2+22k+78}{64}]$. From
Lemma \ref{lem1}, we have $1, 2,...,\frac{k-1}{4}-1\in R$. The
solutions $(11, \frac{k-5}{4}, 1, \frac{5k+39}{16})$,
$(2,1,1,\frac{k+3}{4})$, and $(3,4,1,\frac{k+7}{4})$ show that
$\frac{5k+39}{16},\frac{k+3}{4},\frac{k+7}{4}\in B$. In turn we have
$\frac{4k^2+26k+78}{64}\in R$ by considering $(\frac{5k+39}{16},
\frac{5k+39}{16}, \frac{k+3}{4}, \frac{4k^2+22k+78}{64})$. By
considering $(\frac{k+7}{4},\frac{k+7}{4},\frac{k+3}{4},
\frac{k^2+5k+14}{16})$, we have $\frac{k^2+5k+14}{16}\in R$. Since
$4, \frac{3k+17}{16}\in R$, it follows that
$(4,\frac{k^2+5k+14}{16}, \frac{3k+17}{16}, \frac{4k^2+22k+78}{64})$
is a red solution, a contradiction.

For the lower bound, $\frac{4k^2+22k+78}{64}-1=\frac{4k^2+22k+18}{64}$. Color $[1,\frac{k-1}{4}]\cup [\frac{k^2+5k+14}{16},\frac{4k^2+22k+18}{64}]$ red and
$[\frac{k+3}{4},\frac{k^2+5k-2}{16}]$ blue. For any
solution with $x,y,z$ all blue, we have $4w\geq \frac{k+3}{4}(k+2)$,
so $w$ must be red. Clearly, every monochromatic solution is red.
For every such solution, $w\geq \frac{k^2+5k+14}{16}$. We have
$z\leq \frac{k+1}{4}$, so if $x, y \leq \frac{k+1}{4}$ then
$\frac{k^2+5k+14}{16} \leq w \leq
\lfloor\frac{k^2+3k+2}{16}\rfloor$, a contradiction. Thus, at least
one of $x$ or $y$ must be at least $\frac{k^2+5k+14}{16}$. If $x
\geq \frac{k^2+5k+14}{16}$, then it follows from $y+kz=4w-x$ that$$
\frac{4(k^2+5k+14)}{16}-\frac{4k^2+22k+18}{64} \leq y+kz \leq
\frac{4(4k^2+22k+18)}{64}-\frac{k^2+5k+14}{16},$$
i.e.,$
\frac{12k^2+58k+206}{64} \leq y+kz \leq \frac{12k^2+68k+16}{64}.$
If $z \geq \frac{3k+17}{16}$, then $y \leq
\frac{12k^2+68k+16}{64}-\frac{3k^2+17k}{16}=-2<0$, which is
impossible. If $z \leq \frac{3k+1}{16}$, then $y \geq
\frac{12k^2+58k+206}{64}-\frac{3k^2+k}{16}=\frac{54k+206}{64}>\frac{k-1}{4}$.
Since $y \in R$, we have $y\geq \frac{k^2+5k+14}{16}$. We can assume
that
$x=\frac{4k^2+20k+a}{64},y=\frac{4k^2+20k+b}{64},w=\frac{4k^2+20k+c}{64}$,
with $a,b,c \in[56,2k+18]$. Therefore,$
z=\frac{4w-x-y}{k}=\frac{8k^2+40k+4c-a-b}{64k}=\frac{8k+40+(4c-a-b)/k}{16}.$
Since $\frac{188-4k}{k} \leq \frac{4c-a-b}{k} \leq
\frac{8k-40}{k}$, it follows that $\frac{4c-a-b}{k}\in[-4,7]$, and
$z\in [\frac{8k+36}{64}, \frac{8k+47}{64}]$ is not an integer, and
so there are no monochromatic solutions.
\end{proof}

The proof of Proposition \ref{pro3-10} is:

\begin{proof}
To show the upper bound, we suppose that there is no monochromatic
solution for any $2$-coloring of $[1,\frac{4k^2+18k+68}{64}]$. From
Lemma \ref{lem1}, we have $1, 2,...,\frac{k-2}{4}-1\in R$. The
solutions $(10, \frac{k-6}{4}, 1, \frac{5k+34}{16})$ and
$(1,1,1,\frac{k+2}{4})$ show that $\frac{5k+34}{16},\frac{k+2}{4}\in
B$. In turn we have $\frac{4k^2+18k+68}{64}\in R$ by considering
$(\frac{5k+34}{16}, \frac{5k+34}{16}, \frac{k+2}{4},
\frac{4k^2+18k+68}{64})$. By considering
$(\frac{k+2}{4},\frac{k+2}{4},\frac{k+2}{4}, \frac{k^2+4k+4}{16})$
we have $\frac{k^2+4k+4}{16}\in R$. Since $4, \frac{3k+14}{16}\in
R$, it follows that $(4,\frac{k^2+4k+4}{16}, \frac{3k+14}{16},
\frac{4k^2+18k+68}{64})$ is a red solution, a contradiction.

For the lower bound, $\frac{4k^2+18k+68}{64}-1=\frac{4k^2+18k+4}{64}$. Color $[1,\frac{k-2}{4}]\cup [\frac{k^2+4k+4}{16},\frac{4k^2+18k+4}{64}]$ red and
$[\frac{k+2}{4},\frac{k^2+4k-12}{16}]$ blue. For any solution
with $x,y,z$ all blue, we have $4w\geq \frac{k+3}{4}(k+2)$, and so
$w$ must be red. Therefore, every monochromatic solution is red. For
every such solution, $w\geq \frac{k^2+4k+4}{16}$. Recall $z\leq
\frac{k-2}{4}$. If $x, y \leq \frac{k-2}{4}$, then
$\frac{k^2+4k+4}{16} \leq w \leq \lfloor\frac{k^2-4}{16}\rfloor$, a
contradiction. Thus, at least one of $x$ or $y$ must be at least
$\frac{k^2+4k+4}{16}$. If $x \geq \frac{k^2+4k+4}{16}$, then it
follows from $y+kz=4w-x$ that
$$
\frac{4(k^2+4k+4)}{16}-\frac{4k^2+18k+4}{64} \leq y+kz \leq
\frac{4(4k^2+18k+4)}{64}-\frac{k^2+4k+4}{16},
$$
i.e.,$\frac{12k^2+46k+60}{64} \leq y+kz \leq \frac{12k^2+56k}{64}.$
If $z \geq \frac{3k+14}{16}$, then $y \leq
\frac{12k^2+56k}{64}-\frac{3k^2+14k}{16}=0$, which is impossible. If
$z \leq \frac{3k-2}{16}$, then $y \geq
\frac{12k^2+46k+60}{64}-\frac{3k^2-2k}{16}=\frac{48k+60}{64}>\frac{k-2}{4}$.
Since $y \in R$, we have $y\geq \frac{k^2+4k+4}{16}$. We can assume
that
$x=\frac{4k^2+16k+a}{64},y=\frac{4k^2+16k+b}{64},w=\frac{4k^2+16k+c}{64}$,
with $a,b,c \in[16,2k+4]$. Thus, $
z=\frac{4w-x-y}{k}=\frac{8k^2+32k+4c-a-b}{64k}=\frac{8k+32+(4c-a-b)/k}{16}.$
Since $\frac{56-4k}{k} \leq \frac{4c-a-b}{k} \leq \frac{8k-16}{k}$,
it follows that $\frac{4c-a-b}{k}\in[-4,7]$, and $z\in
[\frac{8k+28}{64}, \frac{8k+39}{64}]$ is not an integer, and so
there are no monochromatic solutions.
\end{proof}

The proof of Proposition \ref{pro3-11} is:
\begin{proof}
To show the upper bound, we suppose that there is no monochromatic
solution for any $2$-coloring of $[1,\frac{4k^2+18k+36}{64}]$. From
Lemma \ref{lem1}, we have $1, 2,...,\frac{k-2}{4}-1\in R$. The
solutions $(6, \frac{k-6}{4}, 1, \frac{5k+18}{16})$ and
$(1,1,1,\frac{k+2}{4})$ show that $\frac{5k+18}{16},\frac{k+2}{4}\in
B$. In turn we have $\frac{4k^2+18k+68}{64}\in R$ by considering
$(\frac{5k+18}{16}, \frac{5k+18}{16}, \frac{k+2}{4},
\frac{4k^2+18k+36}{64})$. By considering
$(\frac{k+2}{4},\frac{k+2}{4},\frac{k+2}{4}, \frac{k^2+4k+4}{16})$
we have $\frac{k^2+4k+4}{16}\in R$. Since $4, \frac{3k+14}{16}\in
R$, one can see that $(2,\frac{k^2+4k+4}{16}, \frac{3k+14}{16},
\frac{4k^2+18k+36}{64})$ is a red solution, a contradiction.

For the lower bound, $\frac{4k^2+18k+68}{64}=\frac{4k^2+18k+4}{64}$. Color $[1,\frac{k-2}{4}]\cup [\frac{k^2+4k+4}{16},\frac{4k^2+18k-28}{64}]$ red and
$[\frac{k+2}{4},\frac{k^2+4k-12}{16}]$ blue. For any solution
with $x,y,z$ all blue, we have $4w\geq \frac{k+3}{4}(k+2)$, so $w$
must be red. It is clear that every monochromatic solution is red.
For every such solution, $w\geq \frac{k^2+4k+4}{16}$. We have $z\leq
\frac{k-2}{4}$, so if $x, y \leq \frac{k-2}{4}$ then
$\frac{k^2+4k+4}{16} \leq w \leq \lfloor\frac{k^2-4}{16}\rfloor$, a
contradiction. Thus, at least one of $x$ or $y$ must be at least
$\frac{k^2+4k+4}{16}$. If $x \geq \frac{k^2+4k+4}{16}$, then it
follows from $y+kz=4w-x$ that
$$
\frac{4(k^2+4k+4)}{16}-\frac{4k^2+18k-28}{64} \leq y+kz \leq
\frac{4(4k^2+18k-28)}{64}-\frac{k^2+4k+4}{16},$$
i.e.,$
\frac{12k^2+46k+92}{64} \leq y+kz \leq \frac{12k^2+56k-128}{64}.$
If $z \geq \frac{3k+14}{16}$, then $y \leq
\frac{12k^2+56k-128}{64}-\frac{3k^2+14k}{16}=-2<0$, which is
impossible. If $z \leq \frac{3k-2}{16}$, then $y \geq
\frac{12k^2+46k+92}{64}-\frac{3k^2-2k}{16}=\frac{48k+92}{64}>\frac{k-2}{4}$.
Since $y \in R$, we have $y\geq \frac{k^2+4k+4}{16}$. We can assume
that
$x=\frac{4k^2+16k+a}{64},y=\frac{4k^2+16k+b}{64},w=\frac{4k^2+16k+c}{64}$,
with $a,b,c \in[16,2k-28]$. Thus, we have$
z=\frac{4w-x-y}{k}=\frac{8k^2+32k+4c-a-b}{64k}=\frac{8k+32+(4c-a-b)/k}{16}.$
Since $\frac{120-4k}{k} \leq \frac{4c-a-b}{k} \leq
\frac{8k-144}{k}$, it follows that $\frac{4c-a-b}{k}\in[-4,7]$, and
hence $z\in [\frac{8k+28}{64}, \frac{8k+39}{64}]$ is not an integer,
and so there are no monochromatic solutions.
\end{proof}

The proof of Proposition \ref{pro3-13} is:
\begin{proof}
To show the upper bound, we suppose that there is no monochromatic
solution for any $2$-coloring of $[1,\frac{k^2+7k+8}{16}]$. From
Lemma \ref{lem1}, we have $1, 2,...,\frac{2k}{16}+1\in R$. The
solutions $(2, 2, 1, \frac{k+4}{4})$ and $(4, 4, 1, \frac{k+8}{4})$
show that $\frac{k+4}{4}, \frac{k+8}{4}\in B$. In turn we have
$\frac{k^2+6k+16}{16}\in R$ by considering $(\frac{k+8}{4},
\frac{k+8}{4}, \frac{k+4}{4}$, $\frac{k^2+6k+16}{16})$. Using $2\in
R$ in $(2,2,2, \frac{2k+4}{4})$, we have $\frac{2k+4}{4}\in B$. We
have $\frac{k^2+7k+8}{16}\in R$ by considering
$(\frac{k+4}{4},\frac{2k+4}{4},\frac{k+4}{4},\frac{k^2+7k+8}{16})$,
and hence $(\frac{k^2+6k+16}{16},\frac{k^2+6k+16}{16},
\frac{2k+16}{16}, \frac{k^2+7k+8}{16})$ is a red solution, a
contradiction.

For the lower bound, $\frac{k^2+7k+8}{16}-1=\frac{k^2+7k-8}{16}$.
Color $[1,\frac{k}{4}]\cup [\frac{k^2+6k+16}{16},\frac{k^2+7k+8}{16}]$ red and $[\frac{k+4}{4},\frac{k^2+6k}{16}]$
blue. For any
solution with $x,y,z$ all blue, we have $4w\geq \frac{k+4}{4}(k+2)$,
so $w$ must be red. Clearly, every monochromatic solution is red.
For every such solution, $w\geq \frac{k^2+6k+16}{16}$. We have
$z\leq \frac{k}{4}$, so if $x, y \leq \frac{k}{4}$ then
$\frac{k^2+6k+16}{16} \leq w \leq \frac{k^2+2k}{16}$, a
contradiction. Thus, at least one of $x$ or $y$ must be at least
$\frac{k^2+6k+16}{16}$. If $x \geq \frac{k^2+6k+16}{16}$, then it
follows from $y+kz=4w-x$ that
$$
\frac{4(k^2+6k+16)}{16}-\frac{k^2+7k-8}{16} \leq y+kz \leq
\frac{4(k^2+7k-8)}{16}-\frac{k^2+6k+16}{16},
$$
i.e.,$\frac{3k^2+17k+72}{16} \leq y+kz \leq \frac{3k^2+22k-48}{16}.
$
If $z \geq \frac{3k+24}{16}$, then $y \leq
\frac{3k^2+22k-48}{16}-\frac{3k^2+24k}{16}=\frac{-2k-48}{16}<0$,
which is impossible. If $z \leq \frac{3k+8}{16}$, then $y \geq
\frac{3k^2+17k+72}{16}-\frac{3k^2+8k}{16}=\frac{9k+72}{16}>\frac{k}{4}$.
Since $y \in R$, we have $y\geq \frac{k^2+6k+16}{16}$. We can assume
that
$x=\frac{k^2+6k+a}{16},y=\frac{k^2+6k+b}{16},w=\frac{k^2+6k+c}{16}$,
with $a,b,c \in[16,k-8]$. Thus, we have$
z=\frac{4w-x-y}{k}=\frac{2k^2+12k+4c-a-b}{16k}=\frac{2k+12+(4c-a-b)/k}{16}.$
Since $\frac{80-2k}{k} \leq \frac{4c-a-b}{k} \leq \frac{4k-64}{k}$,
it follows that $\frac{4c-a-b}{k}\in[-2,3]$, and hence $z\in
[\frac{2k+10}{16}, \frac{2k+15}{16}]$ is not an integer, and so
there are no monochromatic solutions.
\end{proof}

The proof of Proposition \ref{pro3-14} is:
\begin{proof}
To show the upper bound, we suppose that there is no monochromatic
solution for any $2$-coloring of $[1,\frac{k^2+6k+9}{16}]$. From
Lemma \ref{lem1}, we have $1, 2,...,\frac{2k}{16}+1\in R$. The
solutions $(1, 2, 1, \frac{k+3}{4})$ and $(4, 5, 1, \frac{k+9}{4})$
show that $\frac{k+3}{4}, \frac{k+8}{4}\in B$. We have
$\frac{k^2+5k+18}{16}\in R$ by considering $(\frac{k+9}{4},
\frac{k+9}{4}, \frac{k+3}{4}, \frac{k^2+5k+18}{16})$. Using $2\in R$
in $(3,3,2, \frac{2k+6}{4})$, we have $\frac{2k+6}{4}\in B$.
Furthermore, one can see that $\frac{k^2+6k+9}{16}\in R$ by
considering
$(\frac{k+3}{4},\frac{2k+6}{4},\frac{k+3}{4},\frac{k^2+6k+9}{16})$,
and so $(\frac{k^2+5k+18}{16},\frac{k^2+5k+18}{16},
\frac{2k+14}{16}, \frac{k^2+6k+9}{16})$ is a red solution, a
contradiction.

For the lower bound, $\frac{k^2+6k+9}{16}-1=\frac{k^2+6k-7}{16}$. Color $[1,\frac{k-1}{4}]\cup [\frac{k^2+5k+18}{16},\frac{k^2+6k-7}{16}]$ red, $[\frac{k+3}{4},\frac{k^2+5k+2}{16}]$ and
blue. For any
solution with $x,y,z$ all blue, then $4w\geq \frac{k+3}{4}(k+2)$, so
$w$ must be red. Clearly, every monochromatic solution is red. For
every such solution, $w\geq \frac{k^2+5k+18}{16}$. Recall that
$z\leq \frac{k-1}{4}$. If $x, y \leq \frac{k-1}{4}$, then
$\frac{k^2+5k+18}{16} \leq w \leq \frac{k^2+k-2}{16}$, a
contradiction. Thus, at least one of $x$ or $y$ must be at least
$\frac{k^2+5k+18}{16}$. If $x \geq \frac{k^2+5k+18}{16}$, then it
follows from $y+kz=4w-x$ that
$$
\frac{4(k^2+5k+18)}{16}-\frac{k^2+6k-7}{16} \leq y+kz \leq
\frac{4(k^2+6k-7)}{16}-\frac{k^2+5k+18}{16},$$
i.e.,$
\frac{3k^2+14k+79}{16} \leq y+kz \leq \frac{3k^2+19k-46}{16}.$
If $z \geq \frac{3k+21}{16}$, then $y \leq
\frac{3k^2+19k-46}{16}-\frac{3k^2+21k}{16}=\frac{-2k-46}{16}<0$,
which is impossible. If $z \leq \frac{3k+5}{16}$, then $y \geq
\frac{3k^2+14k+79}{16}-\frac{3k^2+5k}{16}=\frac{9k+79}{16}>\frac{k-1}{4}$.
Since $y \in R$, we have $y\geq \frac{k^2+5k+18}{16}$. We can assume
that
$x=\frac{k^2+5k+a}{16},y=\frac{k^2+5k+b}{16},w=\frac{k^2+5k+c}{16}$,
with $a,b,c \in[18,k-7]$. Thus, we have$
z=\frac{4w-x-y}{k}=\frac{2k^2+10k+4c-a-b}{16k}=\frac{2k+10+(4c-a-b)/k}{16}.$
Since $\frac{86-2k}{k} \leq \frac{4c-a-b}{k} \leq \frac{4k-64}{k}$,
it follows that $\frac{4c-a-b}{k}\in[-2,3]$, and hence $z\in
[\frac{2k+8}{16}, \frac{2k+13}{16}]$ is not an integer, and so there
are no monochromatic solutions.
\end{proof}

The proof of Proposition \ref{pro3-15} is:
\begin{proof}
To show the upper bound, we suppose that there is no monochromatic
solution for any $2$-coloring of $[1,\frac{k^2+5k+10}{16}]$. From
Lemma \ref{lem1}, we have $1, 2,...,\frac{k-10}{4}-1\in R$. The
solutions $(1, 1, 1, \frac{k+2}{4})$ and $(5, 5, 1, \frac{k+10}{4})$
show that $\frac{k+2}{4}, \frac{k+10}{4}\in B$. In turn we have
$\frac{k^2+4k+20}{16}\in R$ by considering $(\frac{k+10}{4},
\frac{k+10}{4}, \frac{k+2}{4}, \frac{k^2+4k+20}{16})$. Using $2\in
R$ in $(4,4,2, \frac{2k+8}{4})$ we have $\frac{2k+8}{4}\in B$. Then
$\frac{k^2+5k+10}{16}\in R$ by considering
$(\frac{k+2}{4},\frac{2k+8}{4},\frac{k+2}{4},\frac{k^2+5k+10}{16})$,
and hence $(\frac{k^2+4k+20}{16},\frac{k^2+4k+20}{16},
\frac{2k+12}{16}, \frac{k^2+5k+10}{16})$ is a red solution, a
contradiction.

For the lower bound, $\frac{k^2+5k+10}{16}-1=\frac{k^2+5k-6}{16}$. Color $[1,\frac{k-2}{4}]\cup [\frac{k^2+4k+4}{16},\frac{k^2+5k-6}{16}]$ red and
$[\frac{k+2}{4},\frac{k^2+4k-12}{16}]$ blue. For any solution
with $x,y,z$ all blue, we have $4w\geq \frac{k+2}{4}(k+2)$, and so
$w$ must be red. Thus, every monochromatic solution is red. For
every such solution, $w\geq \frac{k^2+4k+4}{16}$. Recall that $z\leq
\frac{k-2}{4}$. If $x, y \leq \frac{k-2}{4}$ then
$\frac{k^2+4k+4}{16} \leq w \leq \frac{k^2-4}{16}$, a contradiction.
Therefore, at least one of $x$ or $y$ must be at least
$\frac{k^2+4k+4}{16}$. If $x \geq \frac{k^2+4k+4}{16}$, since
$y+kz=4w-x$, then
$$
\frac{4(k^2+4k+4)}{16}-\frac{k^2+5k-6}{16} \leq y+kz \leq
\frac{4(k^2+5k-6)}{16}-\frac{k^2+4k+4}{16},
$$
i.e.,$
\frac{3k^2+11k+22}{16} \leq y+kz \leq \frac{3k^2+16k-28}{16}.$
If $z \geq \frac{3k+18}{16}$ then $y \leq
\frac{3k^2+16k-28}{16}-\frac{3k^2+18k}{16}=\frac{-2k-28}{16}<0$,
which is impossible. If $z \leq \frac{3k+2}{16}$, then $y \geq
\frac{3k^2+11k+22}{16}-\frac{3k^2+2k}{16}=\frac{9k+22}{16}>\frac{k-2}{4}$.
Since $y \in R$, we have $y\geq \frac{k^2+4k+4}{16}$. We can assume
that
$x=\frac{k^2+4k+a}{16},y=\frac{k^2+4k+b}{16},w=\frac{k^2+4k+c}{16}$,
with $a,b,c \in[4,k-6]$. Thus$
z=\frac{4w-x-y}{k}=\frac{2k^2+8k+4c-a-b}{16k}=\frac{2k+8+(4c-a-b)/k}{16}.$
Since $\frac{28-2k}{k} \leq \frac{4c-a-b}{k} \leq \frac{4k-32}{k}$,
it follows that $\frac{4c-a-b}{k}\in[-2,3]$, and hence $z\in
[\frac{2k+6}{16}, \frac{2k+11}{16}]$ is not an integer, and so there
are no monochromatic solutions.
\end{proof}

The proof of Proposition \ref{pro3-17} is:

\begin{proof}
To show the upper bound, we suppose that there is no monochromatic
solution for any $2$-coloring of $[1,\frac{k^2+8k+16}{16}]$. From
Lemma \ref{lem1}, we have $1, 2,...,\frac{k+1}{4}\in R$. The
solution $(2, 2, 1, \frac{k+4}{4})$ shows that $\frac{k+4}{4}\in B$.
Furthermore, $\frac{k^2+6k+8}{16}\in R$ by considering
$(\frac{k+4}{4}, \frac{k+4}{4}, \frac{k+4}{4},
\frac{k^2+6k+8}{16})$. Using $2,4\in R$ in $(4,4,2,
\frac{2k+8}{4})$, we have $\frac{2k+8}{4}\in B$. Then we know that
$\frac{k^2+8k+16}{16}\in R$ by considering
$(\frac{2k+8}{4},\frac{2k+8}{4},\frac{k+4}{4},\frac{k^2+8k+16}{16})$.
It follows that $(\frac{k^2+8k+16}{16},\frac{k^2+8k+16}{16},
\frac{2k+8}{16}, \frac{k^2+6k+8}{16})$ is a red solution, a contradiction.

For the lower bound, $\frac{k^2+8k+16}{16}-1=\frac{k^2+8k}{16}$. Color $[1,\frac{k}{4}]\cup [\frac{k^2+6k+8}{16},\frac{k^2+8k}{16}]$ red and $[\frac{k+4}{4},\frac{k^2+6k-8}{16}]$
blue. For any
solution with $x,y,z$ all blue, we have $4w\geq \frac{k+4}{4}(k+2)$,
and so $w$ must be red. Thus, every monochromatic solution is red.
For every such solution, $w\geq \frac{k^2+6k+8}{16}$. We have $z\leq
\frac{k}{4}$, so if $x, y \leq \frac{k}{4}$ then
$\frac{k^2+6k+8}{16} \leq w \leq \lfloor\frac{k^2+2k}{16}\rfloor$, a
contradiction. Thus, at least one of $x$ or $y$ must be at least
$\frac{k^2+6k+8}{16}$. If $x \geq \frac{k^2+6k+8}{16}$, since
$y+kz=4w-x$, then$$
\frac{4(k^2+6k+8)}{16}-\frac{k^2+8k}{16} \leq y+kz \leq
\frac{4(k^2+8k)}{16}-\frac{k^2+6k+8}{16},$$
i.e.,$
\frac{3k^2+16k+32}{16} \leq y+kz \leq \frac{3k^2+26k-8}{16}.$
If $z \geq \frac{3k+28}{16}$, then $y \leq
\frac{3k^2+26k-8}{16}-\frac{3k^2+28k}{16}=\frac{-2k-8}{16}<0$, which
is impossible. If $z \leq \frac{3k+12}{16}$, then $y \geq
\frac{3k^2+16k+32}{16}-\frac{3k^2+12k}{16}=\frac{4k+32}{16}>\frac{k}{4}$.
Since $y \in R$, we have $y\geq \frac{k^2+6k+8}{16}$. We can assume
that
$x=\frac{k^2+6k+a}{16},y=\frac{k^2+6k+b}{16},w=\frac{k^2+6k+c}{16}$,
with $a,b,c \in[8,2k]$. Thus, we have$
z=\frac{4w-x-y}{k}=\frac{2k^2+12k+4c-a-b}{16k}=\frac{2k+12+(4c-a-b)/k}{16}.$
Since $\frac{32-4k}{k} \leq \frac{4c-a-b}{k} \leq \frac{8k-16}{k}$,
it follows that $\frac{4c-a-b}{k}\in[-3,7]$, and hence $z\in
[\frac{2k+9}{16}, \frac{2k+19}{16}]$ is not an integer, and so there
are no monochromatic solutions.
\end{proof}

The proof of Proposition \ref{pro3-18} is:
\begin{proof}
To show the upper bound, we suppose that there is no monochromatic
solution for any $2$-coloring of $[1,\frac{k^2+7k+12}{16}]$. From
Lemma \ref{lem1}, we have $1, 2,...,\frac{k+1}{4}\in R$. The
solution $(1, 2, 1, \frac{k+3}{4})$ shows that $\frac{k+3}{4}\in B$.
In turn we have $\frac{k^2+5k+6}{16}\in R$ by considering
$(\frac{k+3}{4}, \frac{k+3}{4}, \frac{k+3}{4},
\frac{k^2+5k+6}{16})$. Using $2,3\in R$ in $(3,3,2,
\frac{2k+6}{4})$, we have $\frac{2k+6}{4}\in B$. Then we know that
$\frac{k^2+7k+12}{16}\in R$ by considering
$(\frac{2k+6}{4},\frac{2k+6}{4},\frac{k+3}{4},\frac{k^2+7k+12}{16})$,
and hence $(\frac{k^2+7k+12}{16},\frac{k^2+7k+12}{16},
\frac{2k+6}{16}, \frac{k^2+5k+6}{16})$ is a red solution, a
contradiction.

For the lower bound, $\frac{k^2+7k+12}{16}-1=\frac{k^2+7k-4}{16}$. Color $[1,\frac{k-1}{4}]\cup [\frac{k^2+5k-10}{16},\frac{k^2+7k-4}{16}]$ red and
$[\frac{k+3}{4},\frac{k^2+5k-10}{16}]$ blue. For any solution
with $x,y,z$ all blue, we have $4w\geq \frac{k+3}{4}(k+2)$, and so
$w$ must be red. Thus, every monochromatic solution is red. For
every such solution, $w\geq \frac{k^2+5k+6}{16}$. We have $z\leq
\frac{k-1}{4}$, so if $x, y \leq \frac{k-1}{4}$ then
$\frac{k^2+5k+6}{16} \leq w \leq \lfloor\frac{k^2+k-2}{16}\rfloor$,
a contradiction. Thus at least one of $x$ or $y$ must be at least
$\frac{k^2+5k+6}{16}$. If $x \geq \frac{k^2+5k+6}{16}$, since
$y+kz=4w-x$, then$$
\frac{4(k^2+5k+6)}{16}-\frac{k^2+7k-4}{16} \leq y+kz \leq
\frac{4(k^2+7k-4)}{16}-\frac{k^2+5k+6}{16},$$
i.e.,$
\frac{3k^2+13k+28}{16} \leq y+kz \leq \frac{3k^2+23k-22}{16}.$
If $z \geq \frac{3k+25}{16}$, then $y \leq
\frac{3k^2+23k-22}{16}-\frac{3k^2+25k}{16}=\frac{-2k-22}{16}<0$,
which is impossible. If $z \leq \frac{3k+9}{16}$, then $y \geq
\frac{3k^2+13k+28}{16}-\frac{3k^2+9k}{16}=\frac{4k+28}{16}>\frac{k-1}{4}$.
Since $y \in R$, we have $y\geq \frac{k^2+6k+8}{16}$. We can assume
that
$x=\frac{k^2+5k+a}{16},y=\frac{k^2+5k+b}{16},w=\frac{k^2+5k+c}{16}$,
with $a,b,c \in[6,2k-4]$. Thus$
z=\frac{4w-x-y}{k}=\frac{2k^2+10k+4c-a-b}{16k}=\frac{2k+10+(4c-a-b)/k}{16}.$
Since $\frac{32-4k}{k} \leq \frac{4c-a-b}{k} \leq \frac{8k-28}{k}$,
it follows that $\frac{4c-a-b}{k}\in[-3,7]$, and hence $z\in
[\frac{2k+7}{16}, \frac{2k+17}{16}]$ is not an integer, and so there
are no monochromatic solutions.
\end{proof}

The proof of Proposition \ref{pro3-19} is:
\begin{proof}
To show the upper bound, we suppose that there is no monochromatic
solution for any $2$-coloring of $[1,\frac{k^2+6k+8}{16}]$. From
Lemma \ref{lem1}, we have $1, 2,...,\frac{k-2}{4}\in R$. The
solution $(1, 1, 1, \frac{k+2}{4})$ shows that $\frac{k+2}{4}\in B$.
In turn we have $\frac{k^2+4k+4}{16}\in R$ by considering
$(\frac{k+2}{4}, \frac{k+2}{4}, \frac{k+2}{4},
\frac{k^2+4k+4}{16})$. Using $2\in R$ in $(2,2,2, \frac{2k+4}{4})$,
we have $\frac{2k+4}{4}\in B$. By considering
$(\frac{2k+4}{4},\frac{2k+4}{4},\frac{k+2}{4},\frac{k^2+6k+8}{16})$,
we have $\frac{k^2+6k+8}{16}\in R$, and so
$(\frac{k^2+6k+8}{16},\frac{k^2+6k+8}{16}, \frac{2k+4}{16},
\frac{k^2+4k+4}{16})$ is a red solution.

For the lower bound,
$\frac{k^2+6k+8}{16}-1=\frac{k^2+6k-8}{16}$. We color
$[1,\frac{k-2}{4}]\cup [\frac{k^2+4k+4}{16},\frac{k^2+6k-8}{16}]$ red and $[\frac{k+2}{4},\frac{k^2+4k-12}{16}]$
blue. For any
solution with $x,y,z$ all blue, we have $4w\geq \frac{k+2}{4}(k+2)$,
so $w$ must be red. Clearly, every monochromatic solution is red.
For every such solution, $w\geq \frac{k^2+4k+4}{16}$. Note that
$z\leq \frac{k-2}{4}$. If $x, y \leq \frac{k-2}{4}$, then
$\frac{k^2+4k+4}{16} \leq w \leq \lfloor\frac{k^2-4}{16}\rfloor$, a
contradiction. Thus, at least one of $x$ or $y$ must be at least
$\frac{k^2+4k+4}{16}$. If $x \geq \frac{k^2+4k+4}{16}$, then it
follows from $y+kz=4w-x$ that
$$\frac{4(k^2+4k+4)}{16}-\frac{k^2+6k-8}{16} \leq y+kz \leq
\frac{4(k^2+6k-8)}{16}-\frac{k^2+4k+4}{16},$$ i.e.,
$\frac{3k^2+10k+24}{16} \leq y+kz \leq \frac{3k^2+20k-36}{16}$. If
$z \geq \frac{3k+22}{16}$, then $y \leq
\frac{3k^2+20k-36}{16}-\frac{3k^2+22k}{16}=\frac{-2k-36}{16}<0$,
which is impossible. If $z \leq \frac{3k+6}{16}$, then $y \geq
\frac{3k^2+10k+24}{16}-\frac{3k^2+6k}{16}=\frac{4k+24}{16}>\frac{k-2}{4}$.
Since $y \in R$, we have $y\geq \frac{k^2+4k+4}{16}$. We can assume
that
$x=\frac{k^2+4k+a}{16},y=\frac{k^2+4k+b}{16},w=\frac{k^2+4k+c}{16}$,
with $a,b,c \in[4,2k-8]$. Thus, we have
$z=\frac{4w-x-y}{k}=\frac{2k^2+8k+4c-a-b}{16k}=\frac{2k+8+(4c-a-b)/k}{16}$.
Since $\frac{32-4k}{k} \leq \frac{4c-a-b}{k} \leq \frac{8k-40}{k}$,
it follows that $\frac{4c-a-b}{k}\in[-3,7]$, and hence $z\in
[\frac{2k+5}{16}, \frac{2k+15}{16}]$ is not an integer, and so there
are no monochromatic solutions.
\end{proof}

\section{(For referee) Appendix II: Proof of Theorem \ref{th-Last}}

The proof of Theorem \ref{th-Last} is:

\begin{proof}
Assume, for a contradiction, that there exists a $2$-coloring of
$[1, 5]$ with no monochromatic solution to $x+y+z+kv=(k+1)w$.
Without loss of generality, we assume that $1$ is red. Considering
the solutions $(1,1,1,3,3)$, $(3,3,3,9,9)$, $(1,4,4,9,9)$ and
$(3,4,4,11,11)$, one can see $3,9,4,11$ must be blue, red, blue,
red, respectively, which contradicts to the fact that
$(1,1,9,11,11)$ is a red solution. Hence, we have
$\operatorname{RR}(\mathcal{E}(k,1))\leq 11$, for all $k \in
\mathbb{Z}^+$.

One can easily check that the only valid $2$-colorings (using $r$
for red, $b$ for blue, and assuming that $1$ is red) of $[1,n]$ for
$n=3, 4, 5, 6,7, 8, 9,10$ are in Table $4$.
\begin{table}[H]\label{Talbe2}
\caption{The valid $2$-colorings of $[1,n]$.}
\begin{center}
\begin{tabularx}{10cm}{p{1cm}<{} p{3cm}<{}}
\hline
$n$ & Valid colorings   \\
\hline
$3$  & $rbb, rrb$  \\
$4$  & $rbbr, rbbb, rrbb$ \\
$5$  & $rbbbr, rrbbb, rbbrr, rbbrb, rbbbb$ \\
$6$  & $rrbbbb, rbbbrr, rbbbbr$ \\
$7$  & $rrbbbbr, rbbbbrr, rrbbbbb$ \\
$8$  & $rrbbbbbb, rrbbbbrr, rrbbbbrb, rrbbbbbr$ \\
$9$  & $rrbbbbbbr, rrbbbbbrr$ \\
$10$  & $rrbbbbbbrr$ \\
\hline
\end{tabularx}
\end{center}
\end{table}
If $n\leq 2$, then there is no monochromatic solution to
$\mathcal{E}(5,k,k+j)$. For $n=3$, we consider the valid coloring
$rbb$. If $x,y,z,v$ are all red, then $x+y+z+kv=k+3$. If $x,y,z,v$
are all blue, then $x+y+z+kv\in \{2k+6, 2k+7, 2k+8, 2k+9, 3k+6,
3k+7, 3k+8, 3k+9\}$. If $w$ is red, then $(k+1)w=k+1$. If $w$ is
blue, then $(k+1)w\in \{2k+2, 3k+3\}$. We denote these results by
$R_{x,y,z,v}=\{k+3\}, R_{w}=\{k+1\}$, $B_{w}=\{2k+2, 3k+3\}$,
$B_{x,y,z,v}=\{2k+6, 2k+7, 2k+8, 2k+9, 3k+6, 3k+7, 3k+8, 3k+9\}$. We
see that $B_{x,y,z,v}\cap B_w\neq \emptyset$ only when $k=3 \
(2k+6=3k+3)$, $k=4 \ (2k+7=3k+3)$, $k=5 \ (2k+8=3k+3)$, $k=6 \
(2k+9=3k+3)$.

We now consider the valid coloring $rrb$. If $x, y, z, v$ are all
red, the possible values of $x+y+z+kv$ form the set $\{k+3, k+4,
k+5, k+6, 2k+3, 2k+4, 2k+5, 2k+6 \}$. If $x,y,z,v$ are all blue,
then $x+y+z+kv=3k+9$. If $w$ is red, then $(k+1)w\in \{k+1, 2k+2\}$.
If $w$ is blue, then $(k+1)w=3k+3$. We denote these results by
$R_{x,y,z,v}=\{k+3, k+4, k+5, k+6, 2k+3, 2k+4, 2k+5, 2k+6 \}$, and
$R_{w}=\{k+1, 2k+2\}, \ B_{x,y,z,v}=\{3k+9\}$, $B_{w}=\{3k+3\}$. We
see that $R_{x,y,z,v}\cap R_w\neq \emptyset$ only when $k=1 \
(k+3=2k+2)$, $k=2 \ (k+4=2k+2)$, $k=3 \ (k+5=2k+2)$, $k=4 \
(k+6=2k+2)$. Thus, we have $\operatorname{RR}(\mathcal{E}(k,1))=3$,
for $k=3,4$.

For $n=4$, the valid $2$-colorings are as in Table $5$.
\begin{table}[h]
\caption{The valid $2$-colorings of $[1,4]$.}
\begin{center}
\begin{tabularx}{17cm}{p{3.5cm}<{} p{12.55cm}<{}}
\hline
Valid colorings & Sets  \\
\hline $rbbr$  & $R_{x,y,z,v}=\{ik+j:i=1,4; j=3, 6, 9, 12\};
R_{w}=\{i(k+1):i=1,4\}$
\\& $B_{x,y,z,v}=\{ik+j:i=2,3; j=6, 7, 8, 9\}; B_{w}=\{i(k+1):i=2,3\}$  \\
$rbbb$  & $R_{x,y,z,v}=\{ik+j:i=1; j=3\}; R_{w}=\{i(k+1):i=1\}$
\\& $B_{x,y,z,v}=\{ik+j:i=2,3,4; j=6, 7, 8, 9,10,11,12\}$
\\&$B_{w}=\{i(k+1):i=2,3,4\}$ \\
$rrbb$  & $R_{x,y,z,v}=\{ik+j:i=1,2; j=3,4,5,6\};
R_{w}=\{i(k+1):i=1,2\}$
\\& $B_{x,y,z,v}=\{ik+j:i=3,4; j=9,10,11,12\}; B_{w}=\{i(k+1):i=3,4\}$  \\
\hline
\end{tabularx}
\end{center}
\end{table}

Considering the valid coloring $rbbr$, we see that $R_{x,y,z,v}\cap
R_{w}=\emptyset$ for all $k$, and $B_{x,y,z,v}\cap B_{w}\neq
\emptyset$ only when $k=j-3 \ (3k+3=2k+j)$, $j=6, 7, 8, 9$, and
hence $k=3, 4, 5, 6$. For the valid coloring $rbbb$, we see that
$R_{x,y,z,v}\cap R_{w}=\emptyset$ for all $k$, and $B_{x,y,z,v}\cap
B_{w}\neq \emptyset$ only when $k=j-3 \ (3k+3=2k+j)$, $j=6, 7, 8, 9,
10, 11, 12$, or $k=\frac{j-4}{2} \ (4k+4=2k+j)$, $j=6, 8, 10, 12$,
or $k=j-4 \ (4k+4=3k+j), j=6, 7, 8, 9, 10, 11, 12$. Therefore, $k=1,
2, 3, 4, 5, 6, 7, 8, 9$. For the valid coloring $rrbb$, we see that
$R_{x,y,z,v}\cap R_{w}\neq \emptyset$ only when $k=j-2 \ (2k+2=k+j),
j=3, 4, 5, 6$, and $B_{x,y,z,v}\cap B_{w}\neq \emptyset$ only when
$k=j-4 \ (4k+4=3k+j)$, $j=9, 10, 11, 12$, and hence $k=1, 2, 3, 4,
5, 6, 7, 8$. So, we have $\operatorname{RR}(\mathcal{E}(k,1))=4$ for
$k=5,6$.

For $n=5$, the valid $2$-colorings are as in the following Table
$6$.
\begin{table}[H]
\caption{The valid $2$-colorings of $[1,5]$.}
\begin{center}
\begin{tabularx}{17cm}{p{3.5cm}<{} p{12.55cm}<{}}
\hline
Valid colorings & Sets  \\
\hline $rbbbr$  & $R_{x,y,z,v}=\{ik+j:i=1,5; j=3, 7, 11, 15\};
R_{w}=\{i(k+1):i=1, 5\}$
\\& $B_{x,y,z,v}=\{ik+j:i=2, 3, 4; j=6, 7, 8, 9, 10, 11, 12\}$\\
 &$B_{w}=\{i(k+1):i=2, 3, 4\}$  \\
$rrbbb$  & $R_{x,y,z,v}=\{ik+j:i=1, 2; j=3, 4, 5, 6\};
R_{w}=\{i(k+1):i=1, 2\}$
\\& $B_{x,y,z,v}=\{ik+j:i=3, 4, 5; j=9,10,11,12,13,14,15\}$\\
&$B_{w}=\{i(k+1):i=3,4,5\}$ \\
$rbbrr$  & $R_{x,y,z,v}=\{ik+j:i=1,4,5; j=3,6,10,11,\ldots,15\}$\\
&$R_{w}=\{i(k+1):i=1,4,5\}$
\\& $B_{x,y,z,v}=\{ik+j:i=2,3; j=6,7,8,9\}; B_{w}=\{i(k+1):i=2,3\}$  \\
$rbbrb$  & $R_{x,y,z,v}=\{ik+j:i=1,4; j=3,6,9,12\};
R_{w}=\{i(k+1):i=1,4\}$
\\& $B_{x,y,z,v}=\{ik+j:i=2,3,5; j=6,7,8,9,11,12,13,15\}$\\ &$B_{w}=\{i(k+1):i=2,3,5\}$  \\
$rbbbb$  & $R_{x,y,z,v}=\{ik+j:i=1; j=3\}; R_{w}=\{i(k+1):i=1\}$
\\& $B_{x,y,z,v}=\{ik+j:i=2,3,4,5; j=6,7,\ldots,15\}$\\
& $B_{w}=\{i(k+1):i=2,3,4,5\}$ \\
\hline
\end{tabularx}
\end{center}
\end{table}

Considering the valid coloring $rbbbr$, we see that $R_{x,y,z,v}\cap
R_{w}=\emptyset$ for all $k$, and $B_{x,y,z,v}\cap B_{w}\neq
\emptyset$ only when $k=j-3 \ (3k+3=2k+j)$, $j=6, 7, 8, 9,10,11,12$,
or $k=\frac{j-4}{2} \ (4k+4=2k+j)$, $j=6,8,10,12$, or $k=j-4 \
(4k+4=3k+j)$, $j=6, 7, 8, 9,10,11,12$, and hence $k=1,2,\ldots,9$.

For the valid coloring $rrbbb$, we see that $R_{x,y,z,v}\cap
R_{w}\neq\emptyset$ only when $k=j-2 \ (2k+2=k+j)$, $j=3,4,5,6$, and
$B_{x,y,z,v}\cap B_{w}\neq \emptyset$ only when $k=j-4 \
(4k+4=3k+j)$, $j=9,10,\ldots,15$, $k=\frac{j-5}{2} \ (5k+5=3k+j)$,
$j=9,11,13,15$, or $k=j-5 \ (5k+5=4k+j), j=9,10,\ldots,15$, and hence
$k=1,2,\ldots,11$.

For the valid coloring $rbbrr$, one can see that $R_{x,y,z,v}\cap
R_{w}\neq\emptyset$ only when $k=\frac{j-4}{3} \ (4k+4=k+j)$,
$j=10,13$, or $k=\frac{j-5}{4} \ (5k+5=k+j)$, $j=13$, or $k=j-5 \
(5k+5=4k+j)$, $j=6,10,11,\ldots,15$. Then $B_{x,y,z,v}\cap B_{w}\neq
\emptyset$ only when $k=j-3 \ (3k+3=2k+j)$, $j=6,7,8,9$, and hence
$k=1,2,\ldots,10$.

For the valid coloring $rbbrb$, we see that $R_{x,y,z,v}\cap
R_{w}=\emptyset$, and $B_{x,y,z,v}\cap B_{w}\neq \emptyset$ only
when $k=j-3 \ (3k+3=2k+j)$, $j=6,7,8,9,11,12,13,15$, or
$k=\frac{j-5}{3} \ (5k+5=2k+j)$, $j=8,11$, or $k=\frac{j-5}{2} \
(5k+5=3k+j)$, $j=7,9,11,13,15$. Therefore, we have
$k=1,2,\ldots,10,12$.

For the valid coloring $rbbbb$, one can see that $R_{x,y,z,v}\cap
R_{w}=\emptyset$, and $B_{x,y,z,v}\cap B_{w}\neq \emptyset$ only
when $k=j-3 \ (3k+3=2k+j)$, $j=6,7,\ldots,15$, or $k=\frac{j-4}{2} \
(4k+4=2k+j)$, $j=6,8,10,12,14$, or $k=\frac{j-5}{3}  \ (5k+5=2k+j)$,
$j=8,11,14$, or $k=j-4 \ (4k+4=3k+j)$, $j=6,7,\ldots,15$, or
$k=\frac{j-5}{2} \ (5k+5=3k+j)$, $j=7,9,11,13,15$ or $k=j-5 \
(5k+5=4k+j)$, $j=6,7,\ldots,15$. Thus, $k=1,2,\ldots,12$, and hence
$\operatorname{RR}(\mathcal{E}(k,1))=5$, for $k=1,2,7,8,9$.

For $n=6$, the valid $2$-colorings are as in the following Table
$7$.
\begin{table}[H]
\caption{The valid $2$-colorings of $[1,6]$.}
\begin{center}
\begin{tabularx}{17cm}{p{3.5cm}<{} p{12.55cm}<{}}
\hline
Valid colorings & Sets  \\
\hline $rrbbbb$  & $R_{x,y,z,v}=\{ik+j:i=1,2; j=3,4,5,6\};
R_{w}=\{i(k+1):i=1,2\}$
\\& $B_{x,y,z,v}=\{ik+j:i=3,4,5,6; j=9,10,\ldots,18\}$\\
&$B_{w}=\{i(k+1):i=3,4,5,6\}$  \\
$rbbbrr$  & $R_{x,y,z,v}=\{ik+j:i=1,5,6;
j=3,7,8,11,12,13,15,16,17,18\}$ \\ &$R_{w}=\{i(k+1):i=1,5,6\}; B_{w}=\{i(k+1):i=2,3,4\}$
\\& $B_{x,y,z,v}=\{ik+j:i=2,3,4; j=6,7,...,12\}$ \\
$rbbbbr$  & $R_{x,y,z,v}=\{ik+j:i=1,6; j=3,8,13,18\};
R_{w}=\{i(k+1):i=1,6\}$
\\& $B_{x,y,z,v}=\{ik+j:i=2,3,4,5; j=6,7,\ldots,15\}$\\
&$B_{w}=\{i(k+1):i=2,3,4,5\}$  \\
\hline
\end{tabularx}
\end{center}
\end{table}

Considering the valid coloring $rrbbbb$, we see that
$R_{x,y,z,v}\cap R_{w}\neq\emptyset$ only when $k=j-2 \
(2k+2=k+j),j=3,4,5,6$, and $B_{x,y,z,v}\cap B_{w}\neq \emptyset$
only when $k=j-4 \ (4k+4=3k+j), j=9,10,\ldots,18$, $k=\frac{j-5}{2} \
(5k+5=3k+j), j=9,11,13,15,17$, $k=\frac{j-6}{3} \ (6k+6=3k+j)$,
$j=9,12,15,18$, $k=j-5 \ (5k+5=4k+j)$, $j=9,10,\ldots,18$,
$k=\frac{j-6}{2} \ (6k+6=4k+j)$, $j=10,12,14,16,18$, $k=j-6 \
(6k+6=5k+j), j=9,10,\ldots,18$, and hence $k=1,2,\ldots,14$.

For the valid coloring $rbbbrr$, we see that $R_{x,y,z,v}\cap
R_{w}\neq\emptyset$ only when $k=\frac{j-5}{4} \ (5k+5=k+j)$,
$j=13,17$, $k=\frac{j-6}{5} \ (6k+6=k+j)$, $j=11,16$, $k=j-6 \
(6k+6=5k+j)$, $j=7,8,11,12,13,15,16,17,18$, and $B_{x,y,z,v}\cap
B_{w}\neq \emptyset$ only when $k=j-3 \ (3k+3=2k+j)$,
$j=7,8,11,12,13,15,16,17,18$, $k=\frac{j-4}{2} \ (4k+4=2k+j)$,
$j=8,12,16,18$ and $k=j-4 \ (4k+4=3k+j)$,
$j=7,8,11,12,13,15,16,17,18$, and so $k=1,2,\ldots,15$.

For the valid coloring $rbbbbr$, we see that $R_{x,y,z,v}\cap
R_{w}=\emptyset$ and $B_{x,y,z,v}\cap B_{w}\neq \emptyset$ only when
$k=j-3  \ (3k+3=2k+j), j=6,7,\ldots,15$, $k=\frac{j-4}{2} \
(4k+4=2k+j), j=6,8,10,12,14$, $k=\frac{j-5}{3} \ (5k+5=2k+j),
j=8,11,14$, $k=j-4 \ (4k+4=3k+j), j=6,7,\ldots,15$, $k=\frac{j-5}{2} \
(5k+5=3k+j), j=7,9,11,13,15$, $k=j-5 \ (5k+5=4k+j),
j=6,7,\ldots,15$,thus $k=1,2,\ldots,12$. It follows that
$\operatorname{RR}(\mathcal{E}(k,1))=6$, for $k=10,11,12$.

For $n=7$, the valid $2$-colorings are as in Table $8$.
\begin{table}[ht]
\caption{The valid 2-colorings of $[1,7]$.}
\begin{center}
\begin{tabularx}{17cm}{p{3.5cm}<{} p{12.55cm}<{}}
\hline
Valid colorings & Sets  \\
\hline
$rrbbbbr$  & $R_{x,y,z,v}=\{ik+j:i=1,2,7; j=3,4,5,6,9,10,11,15,16,21\}$\\
& $R_{w}=\{i(k+1):i=1,2,7\}$;$B_{w}=\{i(k+1):i=3,4,5,6\}$
\\& $B_{x,y,z,v}=\{ik+j:i=3,4,5,6; j=9,10,\ldots,18\}$\\
$rbbbbrr$  & $R_{x,y,z,v}=\{ik+j:i=1,6,7;
j=3,8,9,13,14,15,18,19,20,21\}$ \\ &$R_{w}=\{i(k+1):i=1,6,7\}$;$B_{w}=\{i(k+1):i=2,3,4,5\}$
\\& $B_{x,y,z,v}=\{ik+j:i=2,3,4,5; j=6,7,...,15\}$\\
$rrbbbbb$  & $R_{x,y,z,v}=\{ik+j:i=1,2; j=3,4,5,6\};
R_{w}=\{i(k+1):i=1,2\}$
\\& $B_{x,y,z,v}=\{ik+j:i=3,4,5,6,7; j=9,10,\ldots,21\}$\\
&$B_{w}=\{i(k+1):i=3,4,5,6,7\}$  \\
\hline
\end{tabularx}
\end{center}
\end{table}

Considering the valid coloring $rrbbbbr$, we see that
$R_{x,y,z,v}\cap R_{w}\neq\emptyset$ only when $k=j-2 \ (2k+2=k+j),
j=3,4,5,6,9,10,11,15,16,21$, and $B_{x,y,z,v}\cap B_{w}\neq
\emptyset$ only when $k=j-4 \ (4k+4=3k+j), j=9,10,\ldots,18$,
$k=\frac{j-5}{2} \ (5k+5=3k+j), j=9,11,13,15,17$, $k=\frac{j-6}{3} \
(6k+6=3k+j), j=9,12,15,18$, $k=j-5 \ (5k+5=4k+j), j=9,10,...,18$,
$k=\frac{j-6}{2} \ (6k+6=4k+j), j=10,12,14,16,18$, $k=j-6 \
(6k+6=5k+j), j=9,10,\ldots,18$, thus $k=1,2,\ldots,14,19$.

For the valid coloring $rbbbbrr$, we see that $R_{x,y,z,v}\cap
R_{w}\neq\emptyset$ only when $k=\frac{j-6}{5}, j=21$,
$k=\frac{j-7}{6}, j=13,19$, $k=j-7, j=8,9,13,14,15,18,19,20,21$, and
$B_{x,y,z,v}\cap B_{w}\neq \emptyset$ only when $k=j-3 \
(3k+3=2k+j), j=6,7,\ldots,15$, $k=\frac{j-4}{2} \ (4k+4=2k+j),
j=6,8,10,12,14$, $k=\frac{j-5}{3} \ (5k+5=2k+j), j=8,11,14$, $k=j-4
\ (4k+4=3k+j), j=6,7,\ldots,15$, $k=\frac{j-5}{2} \ (5k+5=3k+j),
j=7,9,11,13,15$, $k=j-5 \ (5k+5=4k+j), j=6,7,\ldots,15$, thus
$k=1,2,\ldots,14$.

For the valid coloring $rrbbbbb$, we see that $R_{x,y,z,v}\cap
R_{w}\neq\emptyset$ only when $k=j-2 \ (2k+2=k+j),j=3,4,5,6$, and
$B_{x,y,z,v}\cap B_{w}\neq \emptyset$ only when $k=j-4 \
(4k+4=3k+j), j=9,10,\ldots,21$, $k=\frac{j-5}{2} \ (5k+5=3k+j),
j=9,11,13,15,17,19,21$, $k=\frac{j-6}{3} \ (6k+6=3k+j),
j=9,12,15,18,21$, $k=\frac{j-7}{4} \ (7k+7=3k+j), j=11,15,19$,
$k=j-5 \ (5k+5=4k+j), j=9,10,\ldots,21$, $k=\frac{j-6}{2} \
(6k+6=4k+j), j=10,12,14,16,18,20$, $k=\frac{j-7}{3} \ (7k+7=4k+j),
j=10,13,16,19,21$, $k=j-6 \ (6k+6=5k+j), j=9,10,\ldots,21$,
$k=\frac{j-7}{2} \ (7k+7=5k+j), j=9,11,13,15,17,19,21$, $k=j-7 \
(7k+7=6k+j), j=9,10,\ldots,21$, thus $k=1,2,\ldots,17$. It follows that
$\operatorname{RR}(\mathcal{E}(k,1))=6$, for $k=13,14$.

For $n=8$, the valid $2$-colorings are as in the following Table
$9$.
\begin{table}[ht]
\caption{The valid 2-colorings of [1,8].}
\begin{center}
\begin{tabularx}{17cm}{p{3.5cm}<{} p{12.55cm}<{}}
\hline
Valid colorings & Sets  \\
\hline $rrbbbbbb$  & $R_{x,y,z,v}=\{ik+j:i=1,2; j=3,4,5,6\};
R_{w}=\{i(k+1):i=1,2\}$
\\& $B_{x,y,z,v}=\{ik+j:i=3,4,5,6,7,8; j=9,10,\ldots,24\}$ \\ & $B_{w}=\{i(k+1):i=3,4,5,6,7,8\}$  \\
$rrbbbbrr$  & $R_{x,y,z,v}=\{ik+j:i=1,2,7,8;
j=3,4,5,6,9,10,11,12,15,16,17,18\}$ \\
&$R_{w}=\{i(k+1):i=1,2,7,8\}$;$B_{w}=\{i(k+1):i=3,4,5,6\}$
\\& $B_{x,y,z,v}=\{ik+j:i=3,4,5,6; j=9,10,\ldots,18\}$\\
$rrbbbbrb$  & $R_{x,y,z,v}=\{ik+j:i=1,2,7; j=3,4,5,6,9,10,11,15,16,21\}$\\
& $R_{w}=\{i(k+1):i=1,2,7\}$; $B_{w}=\{i(k+1):i=3,4,5,6,8\}$
\\& $B_{x,y,z,v}=\{ik+j:i=3,4,5,6,8; j=9,10,\ldots,22,24\}$\\
$rrbbbbbr$  & $R_{x,y,z,v}=\{ik+j:i=1,2,8; j=3,4,5,6,10,11,12,17,18,24\}$\\ & $R_{w}=\{i(k+1):i=1,2,8\}$; $B_{w}=\{i(k+1):i=3,4,5,6,7\}$ \\
& $B_{x,y,z,v}=\{ik+j:i=3,4,5,6,7; j=9,10,\ldots,21\}$ \\
\hline
\end{tabularx}
\end{center}
\end{table}

Considering the valid coloring $rrbbbbbb$, we see that
$R_{x,y,z,v}\cap R_{w}\neq\emptyset$ only when
$k=j-2(2k+2=k+j),j=3,4,5,6$, and $B_{x,y,z,v}\cap B_{w}\neq
\emptyset$ only when $k=j-4 \ (4k+4=3k+j), j=9,10,\ldots,24$,
$k=\frac{j-5}{2} \ (5k+5=3k+j), j=9,11,13,15,17,19,21,23$,
$k=\frac{j-6}{3} \ (6k+6=3k+j), j=9,12,15,18,21,24$,
$k=\frac{j-7}{4} \ (7k+7=3k+j), j=11,15,19,23$, $k=\frac{j-8}{5} \
(8k+8=3k+j), j=13,18,23$, $k=j-5 \ (5k+5=4k+j), j=9,10,\ldots,24$,
$k=\frac{j-6}{2} \ (6k+6=4k+j), j=10,12,14,16,18,20,22,24$,
$k=\frac{j-7}{3} \ (7k+7=4k+j), j=10,13,16,19,22$, $k=\frac{j-8}{4}
\ (8k+8=4k+j), j=12,16,20,24$, $k=j-6 \ (6k+6=5k+j), j=9,10,\ldots,24$,
$k=\frac{j-7}{2} \ (7k+7=5k+j), j=9,11,13,15,17,19,21,23$,
$k=\frac{j-8}{3} \ (8k+8=5k+j), j=11,14,17,20,23$, $k=j-7 \
(7k+7=6k+j), j=9,10,\ldots,24$, $k=\frac{j-8}{2} \ (8k+8=6k+j),
j=10,12,14,16,18,20,22,24$, $k=j-8 \ (8k+8=7k+j), j=9,10,\ldots,24$,
thus $k=1,2,\ldots,19$.

For the valid coloring $rrbbbbrr$, we see that $R_{x,y,z,v}\cap
R_{w}\neq\emptyset$ only when $k=j-2 \ (2k+2=k+j),
j=3,4,5,6,9,10,11,12,15,16,17,18$, $k=\frac{j-8}{7} \
(8k+8=k+j),j=15$, $k=\frac{j-7}{5} \ (7k+7=2k+j),j=12,17$, $k=j-8 \
(8k+8=7k+j), j=9,10,11,12,15,16,17,18$ and $B_{x,y,z,v}\cap
B_{w}\neq \emptyset$ only when $k=j-4 \ (4k+4=3k+j), j=9,10\ldots,18$,
$k=\frac{j-5}{2} \ (5k+5=3k+j), j=9,11,13,15,17$, $k=\frac{j-6}{3} \
(6k+6=3k+j), j=9,12,15,18$, $k=j-5 \ (5k+5=4k+j), j=9,10,\ldots,18$,
$k=\frac{j-6}{2} \ (6k+6=4k+j), j=10,12,14,16,18$, $k=j-6 \
(6k+6=5k+j), j=9,10,\ldots,18$, thus $k=1,2,\ldots,16$.

For the valid coloring $rrbbbbrb$, we see that $R_{x,y,z,v}\cap
R_{w}\neq\emptyset$ only when $k=j-2 \ (2k+2=k+j),
j=3,4,5,6,9,10,11,15,16,21$, and $B_{x,y,z,v}\cap B_{w}\neq
\emptyset$ only when $k=j-4 \ (4k+4=3k+j), j=9,10,\ldots,22,24$,
$k=\frac{j-5}{2} \ (5k+5=3k+j), j=9,11,13,15,17,19,21$,
$k=\frac{j-6}{3} \ (6k+6=3k+j), j=9,12,15,18,21,24$,
$k=\frac{j-8}{5} \ (8k+8=3k+j), j=13,18$, $k=j-5 \ (5k+5=4k+j),
j=9,10,\ldots,22,24$, $k=\frac{j-6}{2} \ (6k+6=4k+j),
j=10,12,14,16,18,20,22,24$, $k=\frac{j-8}{4} \ (8k+8=4k+j),
j=12,16,20,24$, $k=j-6 \ (6k+6=5k+j), j=9,10,\ldots,22,24$,
$k=\frac{j-8}{3} \ (8k+8=5k+j), j=11,14,17,20$, $k=\frac{j-8}{2} \
(8k+8=6k+j), j=10,12,14,16,18,20,22,24$, thus $k=1,2,\ldots,20$.

For the valid coloring $rrbbbbbr$, we see that $R_{x,y,z,v}\cap
R_{w}\neq\emptyset$ only when $k=j-2 \
(2k+2=k+j),j=3,4,5,6,10,11,12,17,18,24$, and $B_{x,y,z,v}\cap
B_{w}\neq \emptyset$ only when $k=j-4 \ (4k+4=3k+j), j=9,10,\ldots,21$,
$k=\frac{j-5}{2} \ (5k+5=3k+j), j=9,11,13,15,17,19,21$,
$k=\frac{j-6}{3} \ (6k+6=3k+j), j=9,12,15,18,21$, $k=\frac{j-7}{4} \
(7k+7=3k+j), j=11,15,19$, $k=j-5 \ (5k+5=4k+j), j=9,10,\ldots,21$,
$k=\frac{j-6}{2} \ (6k+6=4k+j), j=10,12,14,16,18,20$,
$k=\frac{j-7}{3} \ (7k+7=4k+j), j=10,13,16,19,21$, $k=j-6 \
(6k+6=5k+j), j=9,10,\ldots,21$, $k=\frac{j-7}{2} \ (7k+7=5k+j),
j=9,11,13,15,17,19,21$, $k=j-7 \ (7k+7=6k+j), j=9,10,\ldots,21$, thus
$k=1,2,\ldots,17,22$. It follows that
$\operatorname{RR}(\mathcal{E}(k,1))=8$, for $k=15,16$.

For $n=9$, the valid $2$-colorings are as in Table $10$.
\begin{table}[h]
\caption{The valid $2$-colorings of $[1,9]$.}
\begin{center}
\begin{tabularx}{17cm}{p{3.5cm}<{} p{12.55cm}<{}}
\hline
Valid colorings & Sets  \\
\hline
$rrbbbbbbr$  & $R_{x,y,z,v}=\{ik+j:i=1,2,9; j=3,4,5,6,11,12,13,19,20\}$\\
& $B_{x,y,z,v}=\{ik+j:i=3,4,5,6,7,8; j=9,10,\ldots,24\}$\\
& $R_{w}=\{i(k+1):i=1,2,9\}$;$B_{w}=\{i(k+1):i=3,4,5,6,7,8\}$ \\
$rrbbbbbrr$  & $R_{x,y,z,v}=\{ik+j:i=1,2,8,9; j=3,4,5,6,10,11,12,13,17,18,19,20,$\\
&$24,25,26,27\}$; $R_{w}=\{i(k+1):i=1,2,8,9\}$\\
& $B_{x,y,z,v}=\{ik+j:i=3,4,5,6,7; j=9,10,\ldots,21\}$  \\
&$B_{w}=\{i(k+1):i=3,4,5,6,7\}$\\
\hline
\end{tabularx}
\end{center}
\end{table}

Considering the valid coloring $rrbbbbbbr$, we see that
$R_{x,y,z,v}\cap R_{w}\neq\emptyset$ only when $k=j-2 \ (2k+2=k+j),
j=3,4,5,6,11,12,13,19,20$, $B_{x,y,z,v}\cap B_{w}\neq \emptyset$
only when $k=j-4 \ (4k+4=3k+j), j=9,10,...,24$, $k=\frac{j-5}{2} \
(5k+5=3k+j), j=9,11,13,15,17,19,21,23$, $k=\frac{j-6}{3} \
(6k+6=3k+j), j=9,12,15,18,21,24$, $k=\frac{j-7}{4}(7k+7=3k+j),
j=11,15,19,23$, $k=\frac{j-8}{5} \ (8k+8=3k+j), j=13,18,23$, $k=j-5
\ (5k+5=4k+j), j=9,10,\ldots,24$, $k=\frac{j-6}{2} \ (6k+6=4k+j),
j=10,12,14,16,18,20,22,24$, $k=\frac{j-7}{3} \ (7k+7=4k+j),
j=10,13,16,19,22$, $k=\frac{j-8}{4}(8k+8=4k+j), j=12,16,20,24$,
$k=j-6 \ (6k+6=5k+j), j=9,10,\ldots,24$, $k=\frac{j-7}{2} \
(7k+7=5k+j), j=9,11,13,15,17,19,21,23$, $k=\frac{j-8}{3} \
(8k+8=5k+j), j=11,14,17,20,23$, $k=j-7 \ (7k+7=6k+j),
j=9,10,\ldots,24$, $k=\frac{j-8}{2} \ (8k+8=6k+j),
j=10,12,14,16,18,20,22,24$, $k=j-8 \ (8k+8=7k+j), j=9,10,\ldots,24$,
thus $k=1,2,\ldots,19$.

For the valid coloring $rrbbbbbrr$, we see that $R_{x,y,z,v}\cap
R_{w}\neq\emptyset$ only when $k=j-2 \
(2k+2=k+j),j=3,4,5,6,10,11,12,13,17,18,19,20,24,25$,
$k=\frac{j-8}{6} \ (8k+8=2k+j),j=20$, $k=\frac{j-9}{8} \
(9k+9=k+j),j=17,25$, $k=j-9 \
(9k+9=8k+j),j=10,11,12,13,17,18,19,20,24,25$, and $B_{x,y,z,v}\cap
B_{w}\neq \emptyset$ only when $k=j-4 \ (4k+4=3k+j), j=9,10,\ldots,21$,
$k=\frac{j-5}{2} \ (5k+5=3k+j), j=9,11,13,15,17,19,21$,
$k=\frac{j-6}{3} \ (6k+6=3k+j), j=9,12,15,18,21$, $k=\frac{j-7}{4} \
(7k+7=3k+j), j=11,15,19$, $k=j-5 \ (5k+5=4k+j), j=9,10,\ldots,21$,
$k=\frac{j-6}{2} \ (6k+6=4k+j), j=10,12,14,16,18,20$,
$k=\frac{j-7}{3} \ (7k+7=4k+j), j=10,13,16,19,21$, $k=j-6 \
(6k+6=5k+j), j=9,10,\ldots,21$, $k=\frac{j-7}{2} \ (7k+7=5k+j),
j=9,11,13,15,17,19,21$, $k=j-7 \ (7k+7=6k+j), j=9,10,\ldots,21$, thus
$k=1,2,\ldots,18,22,23$. It follows that
$\operatorname{RR}(\mathcal{E}(k,1))=9$, for $k=17,18$.

For $n=10$, the valid $2$-colorings are as in Table $11$.
\begin{table}[h]
\caption{The valid $2$-colorings of $[1,10]$.}
\begin{center}
\begin{tabularx}{17cm}{p{3.5cm}<{} p{12.55cm}<{}}
\hline
Valid colorings & Sets  \\
\hline
$rrbbbbbbrr$  & $R_{x,y,z,v}=\{ik+j:i=1,2,9,10; j=3,4,5,6,11,12,13,14,19,20,21,$\\
& $22,27,28,29,30\}$; $R_{w}=\{i(k+1):i=1,2,9,10\}$\\
& $B_{x,y,z,v}=\{ik+j:i=3,4,5,6,7,8; j=9,10,\ldots,24\}$\\
 &$B_{w}=\{i(k+1):i=3,4,5,6\}$ \\

\hline
\end{tabularx}
\end{center}
\end{table}

Considering the valid coloring $rrbbbbbbrr$, we see that
$R_{x,y,z,v}\cap R_{w}\neq\emptyset$ only when $k=j-2(2k+2=k+j),
j=3,4,5,6,11,12,13,19,20,21,22,27,28,29,30$, or $k=\frac{j-10}{9} \
(10k+10=k+j),j=19,28$, and $B_{x,y,z,v}\cap B_{w}\neq \emptyset$
only when $k=j-4 \ (4k+4=3k+j), j=9,10,\ldots,24$, or $k=\frac{j-5}{2}
\ (5k+5=3k+j)$, $j=9,11,13,15,17,19,21,23$, or $k=\frac{j-6}{3} \
(6k+6=3k+j)$, $j=9,12,15,18,21,24$, or $k=\frac{j-7}{4} \
(7k+7=3k+j)$, $j=11,15,19,23$, or $k=\frac{j-8}{5} \ (8k+8=3k+j)$,
$j=13,18,23$, or $k=j-5 \ (5k+5=4k+j)$, $j=9,10,\ldots,24$,
$k=\frac{j-6}{2} \ (6k+6=4k+j)$, $j=10,12,14,16,18,20,22,24$, or
$k=\frac{j-7}{3} \ (7k+7=4k+j)$, $j=10,13,16,19,22$,
$k=\frac{j-8}{4} \ (8k+8=4k+j)$, $j=12,16,20,24$, or $k=j-6 \
(6k+6=5k+j)$, $j=9,10,\ldots,24$, or $k=\frac{j-7}{2} \ (7k+7=5k+j)$,
$j=9,11,13,15,17,19,21,23$, or $k=\frac{j-8}{3} \ (8k+8=5k+j)$,
$j=11,14,17,20,23$, $k=j-7 \ (7k+7=6k+j), j=9,10,\ldots,24$, or
$k=\frac{j-8}{2} \ (8k+8=6k+j)$, $j=10,12,14,16,18,20,22,24$, or
$k=j-8 \ (8k+8=7k+j)$, $j=9,10,\ldots,24$. Thus, $k=1,2,\ldots,19$, and
hence $\operatorname{RR}(\mathcal{E}(k,1))=10$, for $k=19$.
\end{proof}

\end{document}